\documentclass{birkjour}

\usepackage[english]{babel}
\usepackage[utf8x]{inputenc}

\usepackage{amsmath}
\usepackage{amsfonts}
\usepackage{amssymb}
\usepackage{amsthm}

\usepackage{mathtools}

\allowdisplaybreaks

\usepackage{hyperref}
\hypersetup{%
colorlinks=true, %
linkcolor=blue, %
citecolor=blue, %
filecolor=blue, %
urlcolor=blue, %
pdfauthor={R. B. Assuncao, O. H. Miyagaki and J. C. Silva}, %
pdftitle={An existence result for a 
fractional \texorpdfstring{$p$}-Laplacian problem with multiple nonlinearities}, %
pdfsubject={Elliptic partial differential equations}, %
pdfkeywords={Fractional elliptic equations with multiple nolinearities,fractional p-Laplacian operator, 
variational methods}, %
pdfproducer={Latex}, %
pdfcreator={ps2pdf}}

 \newtheorem{thm}{Theorem}[section]
 
 \newtheorem{lem}[thm]{Lemma}
 \newtheorem{prop}[thm]{Proposition}
 \theoremstyle{definition}
 
 \theoremstyle{remark}
 \newtheorem{rem}[thm]{Remark}
  \newtheorem{claim}{Remark}
  
 \numberwithin{equation}{section}

\begin{document}

\title[An existence result for a fractional $p$-Laplacian problem]
{A fractional 
\protect\boldmath${p}$-Laplacian problem 
with multiple critical Hardy-Sobolev nonlinearities}

\author[Assun\c{c}\~{a}o]{Ronaldo B\@. Assun\c{c}\~{a}o}
\address{Departamento de Matem\'{a}tica\,---\,
Universidade Federal de Minas Gerais, 
UFMG \\
Av.~Ant\^{o}nio Carlos, 6627\,---
CEP 30161-970\,---\,Belo Horizonte, MG, Brasil}
\email{ronaldo@mat.ufmg.br}

\author[Miyagaki]{Ol\'{\i}mpio H. Miyagaki}
\thanks{Ol\'{\i}mpio H. Miyagaki was partially supported by 
CNPq/Brasil and INCTMAT/Brasil.}
\address{Departamento de Matem\'{a}tica\,---\,
Universidade Federal de Juiz de Fora, 
UFJF \\
Cidade Universit\'{a}ria\,---\,
CEP 36036-330\,---\,Juiz de Fora, MG, Brasil} 
\email{ohmiyagaki@gmail.com}

\author[Silva]{Jeferson C. Silva}
\address{Departamento de Matem\'{a}tica\,---\,
Universidade Federal de Minas Gerais, 
UFMG \\
Av.~Ant\^{o}nio Carlos, 6627\,---
CEP 30161-970\,---\,Belo Horizonte, MG, Brasil}
\email{jefersoncs@ufmg.br}
\thanks{J. C. Silva have received resarch grants from FAPEMIG/MG and CAPES/Brazil.}
\thanks{Corresponding author: Jeferson C. Silva.}



\date{\today}

\begin{abstract}
In this work, we study the existence of weak solution to the following quasi linear elliptic problem involving the fractional $p$-Laplacian operator, a Hardy potential and multiple critical Sobolev nonlinearities with singularities,
\begin{align*}
(-\Delta_p)^su - \mu \dfrac{\vert u \vert^{p-2} u}{\vert x \vert^{ps}} = \dfrac{\vert u\vert^{p^*_s(\beta)-2}u}{\vert x \vert^{\beta}} + \dfrac{\vert u\vert^{p^*_s(\alpha )-2} u}{\vert x \vert^\alpha }, 
\end{align*}
where
$ x \in \mathbb{R}^N$,
$u\in D^{s,p}(\mathbb{R}^N)$,
$0<s<1$, 
$1<p<+\infty$, 
$N>sp$, 
$0<\alpha<sp$,
$0<\beta<sp$, $\beta\neq\alpha$,
$\mu < \mu_H
 \coloneqq \inf_{u \in D^{s,p}(\mathbb{R}^N) \backslash \{ 0 \}}
[u]_{s,p}^p / \vert\vert u \vert\vert_{s,p}^p > 0$.
To prove the existence of solution to the problem we have to formulate a refined version of the concentration-compactness principle and, as an independent result, we have to show that the extremals for the Sobolev inequality are attained.
\end{abstract}

\maketitle

\section{Introduction and main result}

The fractional $p$-Laplacian operator is a non-linear and non-local operator defined for differentiable functions 
$u \colon \mathbb{R}^N \to \mathbb{R} $ by
\begin{align}
\label{def:plapfrac}
(-\Delta_p)^su(x):=2\lim_{\varepsilon\rightarrow 0^+}\int_{\mathbb{R}^N\setminus B_\varepsilon (x)}\dfrac{\vert u(x)-u(y)\vert^{p-2}(u(x)-u(y))}{\vert x-y\vert^{N+sp}}\,dy,
\end{align}
where
$ x\in \mathbb{R}^N$,
$p \in (1,+\infty)$, $s \in (0,1)$
and
$N > sp$.
The definition~\eqref{def:plapfrac} is consistent, up to a normalization constant dependent only on $N$ and on $s$, 
with the usual definiton of the linear fractional Laplacian operator when $p=2$. In this special case it is simply denoted by $(- \Delta)^s$ and is defined by
\begin{align}
\label{def:2lapfrac}
(- \Delta)^s u(x) & := C(N,s) \,\operatorname{p.v.} \int_{\mathbb{R}^N}
\frac{u(x) - u(y)}{|x-y|^{N+2s}} \, dy,
\end{align}
where 
$\operatorname{p.v.}$ stands for the Cauchy's principal value and the normalization constant is given by
\begin{align*}
C(N,s) & := 2^{2s-1} \pi^{\frac{n}{2}}\,
\frac{\Gamma(\frac{N+2s}{2})}{\Gamma(-s)}.
\end{align*}
In this way it is valid the identity
\begin{align*}
(-\Delta)^s u & = \mathcal{F}^{-1} \left( |\xi|^{2s} \left( \mathcal{F}u \right) \right),
\end{align*}
where $\xi \in \mathbb{R}^N$,
$u \in \mathcal{S}(\mathbb{R}^N)$,
the class of Schwartz differentiable functions with repid decay,
and 
\begin{align*}
\mathcal{F}u(\xi) & = \int_{\mathbb{R}^N} \exp(-2\pi i x \cdot \xi) u(x) \, dx
\end{align*}
denotes the Fourier transform of $u$.

For more basic informations about the fractional $p$-Laplacian operator we cite the article by Di Nezza, Palatucci and Valdinoci~\cite{MR2944369} and the book by 
Molica Bisci, R\u{a}dulescu, Servadei~\cite{MR3445279},
as well as the references therein; 
for some motivations in physics, chemistry, and economy
that lead to the study of this kind of operator we mention the article by Caffarelli~\cite{MR2354493}.

Non-local problems involving the fractional $p$-Laplacian operator ~$(- \Delta_p)^s$ have received the attention of several authors in the last decade, mainly in the case $p=2$
and in the cases where the nonlinearities have pure polynomial growth involving subcritical exponents (in the sense of the Sobolev embeddings). For example, this operator leads naturally to the class of quasi linear problems
\begin{align}
\label{eq:probquaselin}
\begin{cases}
(- \Delta)_p^s u = f(x,u) 
& x \in \Omega \subset \mathbb{R}^N\\
u=0 & x \in \mathbb{R}^N\setminus \Omega,
\end{cases}
\end{align}
where $\Omega \subset \mathbb{R}^N$ is a domain. 
Nowadays there exists an extensive and ever growing literature about the class of quasi linear 
problems~\eqref{eq:probquaselin} in the case where $\Omega$ is bounded and with Lipschitz boundary. In particular,
we cite 
Franzina and Palatucci~\cite{MR3307955} 
and 
Lindgren and Lindqvist~\cite{MR3148135}
for problems involving $p$-eigenvalues; 
Di Castro, Kuusi and Palatucci~\cite{MR3542614} 
and
Iannizzotto, Mosconi and Squassina~\cite{MR3593528,MR3470672}
for regularity theory;
Iannizzotto, Liu, Perera, and 
Squassina~\cite{MR3483598},
Molica Bisci, R\u{a}dulescu, and Servadei~\cite{MR3445279}
and Servadei and Valdinoci~\cite{MR2879266} 
for the theory of existence of solutions in the case of 
nonlinearities with pure polynomial growth involving
subcritical exponents;
Alves and Miyagaki~\cite{MR3494494},
Fiscella, Molica Bisci and Servadei~\cite{MR3446947},
Servadei and Valdinoci~\cite{MR3379042}
for the theory of existence of solutions in the case of nonlinearities with pure polynomial growth involving critical exponents.
Moreover, great attention has been given to the study of existence of solutions to
nonlocal problems with the Hardy potential
and also with other types of nonlinearities; for these cases, we cite
Abdellaoui, Peral and Primo~\cite{MR3186917},
Barrios, Medina and Peral~\cite{MR3231059},
Cotsiolis and Tavoularis~\cite{MR2064421}
and Yang and Wu~\cite{MR3709036}, as well as the references therein.

In what follows, we mention an interesting class of quasi linear elliptic problems in the general class of problems~\eqref{eq:probquaselin}; 
more precisely, we consider problems with multiple critical nonlinearities in the sense of the Sobolev embeddings and also a nonlinearity of the Hardy type, which consistently appears on the side of the nonlocal operator. 
Fillippucci, Pucci and Robert~\cite{MR2498753} 
considered the quasi linear elliptic problem
\begin{align}
\label{eq:fpr}
-\Delta_pu-\mu\dfrac{u^{p-1}}{\vert x\vert^p}=u^{p^*-1}+\dfrac{u^{p^*(\alpha)-1}}{\vert x\vert^\alpha} \qquad (x\in\mathbb{R}^N),
\end{align}
where $\Delta_pu=\operatorname{div} (\vert\nabla u\vert^{p-2}\nabla u)$ is the $p$-Laplacian operator, 
$N\geqslant 2$ is an integer, 
$p\in (1,N)$ 
$\alpha \in (0,p)$,
$p^*(\alpha)=[p(N-\alpha)](N-p)$, 
$p^* = p^*(0) = Np/(N-p)$
and
$\mu$ is a real parameter.
The combination of two nonlinearities leads to some serious difficulties and subtlities to problem~\eqref{eq:fpr}. 
When only one nonlinearity appears with critical exponent, several results about the existence of weak solutions are already known. In general, these weak solutions are radially symmetric with respect to some point. The common strategy to obtain a solution to problem~\eqref{eq:fpr} consists in the construction of solutions as critical points of the energy functional naturally associated to this class of problems, since they have variational structure. To do this, the authors used a version of the mountain pass theorem due to Ambrosetti e Rabinowitz. However, since the problem is invariant under the action of the group of conformal transformations 
$u \mapsto u_r(x):= r^{(N-p)/p} u(rx)$,
the mountain pass theorem yields only Palais-Smale sequences
and not necessarily critical points for the energy functional. So, an important step in the proof of their existence result consists in showing a refined version of the concentration-compactness principle in order to better understand the behavior of the Palais-Smale sequences. The main difficulty is that there is an asymptotic competition between the energy carried by the two critical nonlinearities. If one of them dominates the other, then there is the anihilation of the weaker one; in this case, the limit of the Palais-Smale sequence is a weak solution of a problem involving only one critical nonlinearity. Of course, in this case we do not obtain a weak solution to problem~\eqref{eq:fpr}. Therefore, the crucial point consists in avoiding the domination of one nonlinearity over the other.

Afterwards, 
Ghoussoub and Shakerian~\cite{MR3366777}
considered the quasi linear nonlocal elliptic problem
\begin{align}
(-\Delta)^su-\mu\dfrac{u}{\vert x\vert^{2s}}=\vert u\vert^{2^*-2}u + \dfrac{\vert u\vert^{2^*_s(\beta)-2}u}{\vert x\vert^{\beta}} \qquad (x\in\mathbb{R}^N)\label{2.2}.
\end{align}
This problems generalizes the one studied by 
Filippucci, Pucci and Robert to the case of nonlocal operators; 
more specificaly, to the fractional Laplacian operator with
$p=2$. Besides the above mentioned difficulties caused by the presence of multiple critical nonlinearities,
in the case of problem~\eqref{2.2} there exist additional difficulties. To show the existence of weak solution, the authors considered an idea proposed by Caffarelli and Silvestre~\cite{MR2354493} that uses the harmonic extension of the fractional Laplacian operator to the upper half-space
$\mathbb{R}_+^{n+1}$, changing the given nonlocal problem  to a local problem with Neumann boundary condition.

Recently, Chen~\cite{MR3762809} considered the quasi linear nonlocal elliptic problem
\begin{align}
(-\Delta)^su-\mu\dfrac{u}{\vert x\vert^{2s}}=\dfrac{\vert u\vert^{2^*_s(\alpha)-2}u}{\vert x\vert^{\alpha}}+\dfrac{\vert u\vert^{2^*_s(\beta)-2}u}{\vert x\vert^{\beta}} \qquad (x\in\mathbb{R}^N)\label{2.1}.
\end{align}
This problem generalizes the problem studied by Ghoussoub and Shakerian, still in the case $p=2$ but for the case where both nonlinearities have singularities at the origin.
Again, the basic strategy used by the author to show the existence of weak, positive solution to problem~\eqref{2.1}
was the use of the harmonic extension of the fractional Laplacian proposed by Caffarelli and Silvestre~\cite{MR2354493} as well as the mountain pass theorem and the concentration-compactness principle.

Motivated by the several results above mentioned, in this work we consider the quasi linear elliptic problem involving the fractional $p$-Laplacian problem with multiple critical nonlinearities with singularities at the origin and a Hardy term,
\begin{align}
(-\Delta_p)^su-\mu\dfrac{\vert u\vert^{p-2}u}{\vert x\vert^{ps}}=\dfrac{\vert u\vert^{p_s^*(\beta)-2}u}{\vert x\vert^{\beta}}+\dfrac{\vert u\vert^{p_s^*(\alpha)-2}u}{\vert x\vert^{\alpha}} \qquad (x\in\mathbb{R}^N) \label{1.1}
\end{align}
where
$0<s<1$, 
$1<p<+\infty$, 
$N>sp$, 
$0<\alpha<sp$,
$0<\beta<sp$, $\beta\neq\alpha$,
$\mu < \mu_H$ 
(the constant $\mu_H$ is defined below)
and
$ p^*_s(\alpha)=(p(N-\alpha)/(N-ps)$;
in particular, if $\alpha =0$ then
$p_s^*(0)=p_s^*=Np/(N-p)$.

The choice of the space function where we look for the solutions to problems with variational structure such as problem~\eqref{1.1} is an important step in its study.
Let $\Omega \subset \mathbb{R}^N$ be an open, bounded subset with differentiable boundary. We consider tacitly that all the functions are Lebesgue integrable and we introduce the fractional Sobolev space
\begin{align*}
W_0^{s,p}(\Omega) 
& \coloneqq \left\{ 
u \in L_{\operatorname{loc}}^1 (\mathbb{R}^N) 
\colon [u]_{s,p} < +\infty; \; u \equiv 0 \text{ a.e. } \mathbb{R}^N \backslash \Omega\right\}
\end{align*}
and the fractional homogeneous Sobolev space
\begin{align*}
D^{s,p}(\mathbb{R}^N)
& \coloneqq \left\{ u\in L^{p^*_s}(\mathbb{R}^N) 
\colon [u]_{s,p}<\infty \right\}
\supset W_0^{s,p}(\Omega).
\end{align*}
In these definitions, the symbol
$[u]_{s,p}$ stands for the Gagliardo seminorm, defined by
\begin{align*}
u & \longmapsto [u]_{s,p}=\left(\int_{\mathbb{R}^N}\int_{\mathbb{R}^N}\dfrac{\vert u(x)-u(y)\vert^{p}}{\vert x-y\vert^{N+sp}}\,dx\,dy\right)^{1/p} \quad 
(u\in C^{\infty}_0(\mathbb{R}^N)).
\end{align*}
 
For $ 1 < p < +\infty$, 
the function spaces 
$W_0^{s,p}(\Omega)$ 
and 
$D^{s,p}(\mathbb{R}^N)$ 
are separable, reflexive Banach spaces
with respect to the Gagliardo  seminorm 
$[ \, \cdot \,]_{s,p}$.
These spaces can also be understood as the respective completions of the spaces of differentiable functions with compact support
$C_0^{\infty}(\Omega)$
and  
$C_0^{\infty}(\mathbb{R}^N)$
with respect to 
$[ \, \cdot \,]_{s,p}$; 
see, for example, 
Brasco, Mosconi and Squassina~\cite{MR3461371}.
The topological dual of the space 
$W_0^{s,p}(\Omega)$ 
is denoted by 
$W^{-s,p'}(\Omega)$ 
where $1/p + 1/p' = 1$ 
or by $(W_0^{s,p}(\Omega))'$,
with the corresponding duality product 
$\langle \,\cdot\, , \,\cdot \, \rangle \colon W^{-s,p'}(\Omega) \times W_0^{s,p}(\Omega) \to \mathbb{R}$.
Due to the reflexivity of the space, the weak convergence and the weak${}^*$ convergence in $W^{-s,p'}(\Omega)$ coincide.
Moreover, in the Sobolev space $D^{s,p}(\mathbb{R}^N)$,
the space where we look for solutions to problem~\eqref{1.1}, 
the Gagliardo seminorm 
$[\,\cdot\,]_{s,p}$ is in fact a norm and 
$(D^{s,p}(\mathbb{R}^N); [\,\cdot\,]_{s,p})$
is an uniformly convex Banach space.

The variational structure of problem~\eqref{1.1} can be established with the help of the following version of the Hardy-Sobolev inequality, which can be found in the paper by
Chen, Mosconi and Squassina~\cite{MR3861730}. 

Let $0 < s < 1$, $1 < p < +\infty$ and 
$0 \leqslant \alpha < sp < N$.  Then there exists a positive constant $C \in \mathbb{R}_{+}$ such that
\begin{align}
\label{desig:hs}
\left( 
\int_{\Omega} \frac{\vert u \vert^{p_{\alpha}^*}}{\vert x \vert^{\alpha}} \,dx \right)^{1/p_{\alpha}^*}
& \leqslant C 
\left( 
\int_{\mathbb{R}^N} \int_{\mathbb{R}^N} 
\frac{\vert u(x) - u(y) \vert^p}{\vert x - y \vert^{N + ps}} \,dx \,dy \right)^{1/p}
\end{align}
for every $u \in W_0^{s,p}(\Omega)$.
The parameter $p_s^*(\alpha)$
is the critical fractional exponent of the Hardy-Sobolev embeddings 
$D^{s,p}(\mathbb{R}^N)\hookrightarrow L^p(\mathbb{R}^N;\vert x \vert^{-sp})$
where the Lebesgue space
$L^p(\mathbb{R}^N;\vert x \vert^{-sp})$
is equipped with the norm
\begin{align*}
\vert\vert u \vert\vert_{L^p(\mathbb{R}^N;\vert x \vert^{-sp})} 
& \coloneqq
\left(\int_{\mathbb{R}^N} \frac{\vert u \vert^p}{\vert x \vert^{sp}}\,dx \right)^{1/p}.
\end{align*}
Indeed, the embeddings 
$W_0^{s,p}(\Omega) \hookrightarrow L^q(\Omega; \vert x \vert^{\alpha})$
are continuous for 
$ 0 \leqslant \alpha \leqslant ps$ and for 
$1 \leqslant q \leqslant p_s^*(\alpha)$;
and these embeddings are compact for
$1 \leqslant q < p_s^*(\alpha)$.
Moreover, the best constants of these embeddings are positive numbers, that is,
\begin{align}
\mu_H
& \coloneqq \inf_{\substack{u \in D^{s,p}(\mathbb{R}^N) \\ u \neq 0}}
\frac{[u]_{s,p}^p}{\| u \|_{L^p(\mathbb{R}^N;\vert x \vert^{-sp})}^{p}}.\label{muH}
\end{align}

The functional $u \longmapsto (1/p) [ u ]_{s,p}^{p}$ is convex and is belongs to the class 
$C^1(D^{s,p}(\Omega);\mathbb{R})$,
so that for every function
$ u \in W_0^{s,p}(\Omega)$, its subdifferential is exactly
$(-\Delta_p)^s u$, that is, 
the unique element of the topological dual space
$W^{-s,p'}(\Omega)$ such that
\begin{align*}
\left\langle (-\Delta_p)^s u, \phi \right\rangle
& = \int_{\mathbb{R}^N}
\int_{\mathbb{R}^N} \frac{J_p(u(x) - u(y)) (\phi(x) - \phi(y))}{\vert x - y \vert^{N+ps}} \,dx \,dy;\, \forall\, \phi \in W_0^{s,p}(\Omega).
\end{align*}
In the previous formula, by way of simplicity we introduced the notation:
given $1 < m < + \infty$, we define the function
$J_m \colon \mathbb{R} \to \mathbb{R}$ by
$J_m(t)=\vert t \vert^{m-2}t$.

Now we can define precisely the notion of weak solution to problem~\eqref{1.1}.  
We say that the function 
$u\in D^{s,p}(\mathbb{R}^N)$ 
is a weak solution to problem~\eqref{1.1} if
\begin{align*}
&\int_{\mathbb{R}^{N}}\int_{\mathbb{R}^{N}}
\dfrac{J_p(u(x)-u(y))(\varphi(x)-\varphi(y))}{\vert x-y\vert^{N+sp}}\,dx\,dy
-\mu\int_{\mathbb{R}^N}\dfrac{J_p(u)\varphi(x)}{\vert x\vert^{ps}}\,dx \\
& \qquad = \int_{\mathbb{R}^N}\dfrac{J_{p_s^*(\beta)}(u)\varphi(x)}{\vert x\vert^{\beta}}\,dx 
+ \int_{\mathbb{R}^N}
\dfrac{J_{p_s^*(\alpha)}(u)\varphi(x)}{\vert x\vert^{\alpha}}\,dx 
\end{align*}
for every function 
$\varphi\in D^{s,p}(\mathbb{R}^N)$.

By the notation introduced and by the results above mentioned, we see that a weak solution to problem~\eqref{1.1} corresponds to a critical point to the functional
$\Phi \colon \mathbb{R}^N \to \mathbb{R}$
defined by
\begin{align}
\label{eq:funcenergia}
\Phi(u) & \coloneqq
\frac{1}{p} [u]_{s,p}^p - \frac{\mu}{p} \int_{\mathbb{R}^N}
\frac{\vert u \vert^p}{\vert x \vert^{ps}} \, dx
- \frac{1}{p_s^*(\beta)} \int_{\mathbb{R}^N}
\frac{\vert u \vert^{p_x^*(\beta)}}{\vert x \vert^{\beta}} \, dx \nonumber \\
& \quad - \frac{1}{p_s^*(\alpha)} \int_{\mathbb{R}^N}
\frac{\vert u \vert^{p_x^*(\alpha)}}{\vert x \vert^{\alpha}} \, dx,
\end{align}
named energy functional.
In fact, for the parameters in the intervals already specified, we have 
$\Phi \in C^1(D^{s,p}(\mathbb{R}^N);\mathbb{R})$   
\begin{align*}
\left\langle \Phi'(u), \phi \right\rangle
& = \int_{\mathbb{R}^{N}}\int_{\mathbb{R}^{N}}
\dfrac{J_p(u(x)-u(y))(\varphi(x)-\varphi(y))}{\vert x-y\vert^{N+sp}}\,dx\,dy\\
&\qquad -\mu\int_{\mathbb{R}^N}\dfrac{J_p(u)\varphi(x)}{\vert x\vert^{ps}}\,dx - \int_{\mathbb{R}^N}\dfrac{J_{p_s^*(\beta)}(u)\varphi(x)}{\vert x\vert^{\beta}}\,dx \\
& \qquad - \int_{\mathbb{R}^N}\dfrac{J_{p_s^*(\alpha)}(u)\varphi(x)}{\vert x\vert^{\alpha}}\,dx,\quad \forall \varphi\in D^{s,p}(\mathbb{R}^N).
\end{align*}
In other terms, 
$u \in D^{s,p}(\mathbb{R}^N)$ 
is a weak solution to problem~\eqref{1.1} if, and only if,
$\Phi'(u) = 0$.

Now we can state our result.

\begin{thm}
Let 
$0<s<1$, 
$1<p<+\infty$, 
$N>sp$, 
$0<\alpha<sp$,
$0<\beta<sp$, $\beta\neq\alpha$,
$\mu < \mu_H$.
Then there exists a weak solution
$ u \in  D^{s,p}(\mathbb{R}^{N})$ 
to problem~\eqref{1.1}.
\label{teo1}
\end{thm}

The proof of Theorem~\ref{teo1} follows several ideas that have appeared in the papers by  
Filippucci, Pucci, and Robert~\cite{MR2498753},
by Ghoussoub and Shakerian~\cite{MR3366777},
and also by Chen~\cite{MR3762809}. 
However, since we consider the case
$1 < p < + \infty$,
we cannot apply the harmonic extension of the fractional Laplacian as described by Caffarelli and Silvestre~\cite{MR2354493} because this idea is valid only in the case $p=2$. 
Moreover, since we consider the whole space $\mathbb{R}^N$ and since problem~\eqref{1.1} contains critical nonlinearities in the sense of the Hardy-Sobolev embeddings, 
it follows that the Hardy-Sobolev embedding 
$D^{s,p}(\mathbb{R}^N)\hookrightarrow L^p(\mathbb{R}^N;\vert x \vert^{-sp})$
is non compact. 
This poses several difficulties to prove that bounded Palais-Smale in the reflexive Banach space $D^{s,p}(\mathbb{R}^N)$ have at least a subsequence that converges strongly to a nontrivial function in this space. Clear enough, the presence of multiple Sobolev critical nonlinearities also contributes to the difficulties in the proof of the theorem. Moreover,
due to the presence of a Hardy potential, with the parameters in the already specified intervals, the functional 
$u \longmapsto ([u]_{s,p}^p - \mu \int_{\mathbb{R}^N} |u|^{p}/|x|^{sp}\,dx) ^{1/p}$ does not define a norm in the Sobolev space $D^{s,p}(\mathbb{R}^N)$, although it can be compared to a suitable norm (see Goyal~\cite{MR3772129} and Filippucci, Pucci and Robert~\cite{MR2498753}); 
as a consequence, the energy functional $\Phi$ is not lower semicontinuous. Based on some estimates proved by 
Brasco, Mosconi, and Squassina~\cite{MR3461371},
by Xiang, B. Zhang, and X. Zhang~\cite{MR3667062}, 
and by
Brasco, Squassina, and Yang~\cite{MR3732174}, 
we managed to overcome these difficulties and prove a refined version of the concentration-compactness principle.

Finally, we should mention that Theorem~\ref{teo2} below is crucial in the proof of Theorem~\ref{teo1}. 

\begin{thm}
The best Hardy constant, defined by  
\begin{align}
\frac{1}{K(\mu,\alpha)}
=\inf_{\substack{u \in D^{s,p}(\mathbb{R}^N) \\ u \neq 0}}
\dfrac{[u]_{s,p}^{p} -\mu\displaystyle\int_{\mathbb{R}_N}\dfrac{\vert u\vert^p}{\vert x\vert^{ps}}\;dx}{\left(\displaystyle\int_{\mathbb{R}^N}\dfrac{\vert u\vert^{p^*_s(\alpha)}}{\vert x\vert^{\alpha}}\;dx\right)^{\frac{p}{p_s^*(\alpha)}}}, \label{6}
\end{align} 
is attained by a nontrivial function $u\in D^{s,p}(\mathbb{R}^N)$.\label{teo2}
\end{thm}
The paper is divided in several sections. 
In section~\ref{sec:prelim}, we use the mountain pass theorem to show the existence of suitable Palais-Smale sequences;
in section~\ref{sec:structureps}, we study the behavior of these Palais-Smale sequences;
in section~\ref{sec:proofteo1}, we prove Theorem~\ref{teo1};
and in section~\ref{sec:extremals}, we prove Theorem~\ref{teo2}.

\section{Preliminary results}
\label{sec:prelim}

In this section we present some preliminary results that will be usefull in the proof of Theorem~\ref{teo1}. By the definition of $\mu_H$, the following inequality is valid,
\begin{align}
\mu_H 
& \leqslant\dfrac{[u]_{s,p}^p}{\displaystyle\int_\mathbb{R^N} 
\dfrac{\vert u \vert^p}{\vert x \vert^{ps}} \, dx}
\quad \text{for all } u \in D^{s,p}(\mathbb{R}^N)
\setminus \{0\}.\label{desg.hardy}
\end{align}
It is well known that the sharp constant $\mu_H$ is attained; the proof of this claim can be found in paper by Frank and Seiringer~\cite{MR2469027}. 
For $0 < \mu < \mu_H $, it follows from 
inequality~\eqref{desg.hardy} that 
\begin{equation}
\Vert u \Vert \coloneqq \left([u]_{s,p}^p-\mu\int_{\mathbb{R}^N}\frac{\vert u\vert^p}{\vert x \vert^{ps}}dx\right)^{\frac{1}{p}}\label{def.norma.dsmu}
\end{equation}
is well defined in the space $D^{s,p}(\mathbb{R}^N)$. 
Note that $\Vert \cdot \Vert$ is comparable to the Gagliardo norm 
$[\,\cdot\,]_{s,p}$; to see this, it is sufficient to use the pair of inequalities
\begin{equation}
\left(1-\frac{\mu_+}{\mu_1}\right)[u]_{s,p}^p\leqslant\Vert u\Vert^p\leqslant\left( 1 +\frac{\mu_-}{\mu_1}\right)[u]_{s,p}^p, \label{eq4fp}
\end{equation}
valid for all
$u \in D^{s,p}(\mathbb{R}^N)$
where $\mu_+ \coloneqq \max\{\mu ,0\}$ 
and 
$\mu_- \coloneqq \max\{-\mu ,0\}$. 
By combining the Hardy inequality proved in~\cite{MR2469027}
together with the Sobolev inequality proved in 
Brasco, Mosconi and Squassina~\cite{MR3461371},
we obtain an inequality of the Hardy-Sobolev type. 
Indeed, for $ 0 < \alpha < ps $, by the Hölder inequality and 
by the fractional versions of the Hardy and the Sobolev 
inequalities, the embedding 
$D^{s,p}(\mathbb{R^N})
\hookrightarrow 
L^{p_s^*(\alpha)}(\mathbb{R^N},\vert x \vert^{-\alpha})$ 
is continuous. 
Using the sharp constant of this embddin, we defined the constant $K(\mu,\alpha)$ in~\eqref{6} 
and in section~\ref{sec:extremals} we show the Theorem~\ref{teo2}.

Now we recall the definition of the energy functional naturally associated to the variational problem~\eqref{1.1} and rewrite it in the form
\begin{align*}
\Phi (u)=\frac{\Vert u \Vert^p}{p} - \frac{1}{p^*_s(\beta)}\displaystyle\int_{\mathbb{R}^N}\dfrac{|u|^{p^*_s(\beta)}}{\vert x \vert^{\beta}}dx -\frac{1}{p^*_s(\alpha)}\int_{\mathbb{R}^N}\frac{|u|^{p^*_s(\alpha)}}{\vert x \vert^{\alpha}}dx.
\end{align*}
Using the fractional versions of the Hardy and the Hardy-Sobolev inequalities above mentioned, it is easy to show that the functional $\Phi$ is well defined in the space $D^{s,p}(\mathbb{R}^N)$.

Recall that the sequence $\{ u_n \} \subset D^{s,p}(\mathbb{R}^N)$ is a Palais-Smale sequence for the energy functional $\Phi$ at the level $c \in \mathbb{R}$, 
in short $(PS)_c$, if $\Phi(u_n) \to c$ and $\langle \Phi'(u_n), \phi \rangle \to 0$ for every $\phi \in D^{s,p}(\mathbb{R}^N)$. Our first preliminary result concerns the existence of the Palais-Smale sequence for the energy functional $\Phi$.

\begin{prop}
Let $ \mu \in [0,\mu_H)$ and$ \alpha\in [0,sp)$.
Then there exists a sequence 
$\{u_n\}_n\subset D^{s,p}(\mathbb{R}^N)$ 
such that 
\begin{align*}
\displaystyle\lim_{n\rightarrow +\infty}\Phi(u_n)=c
\text{ and } 
\displaystyle\lim_{n\rightarrow +\infty}\Phi'(u_n)=0
\text{ strongly in }(D^{s,p}(\mathbb{R}^N))',
\end{align*}
where $ 0 < c < c_*$ and is defined as being
\begin{align}
\min\left\lbrace
\left(\frac{p_s^*(\beta)-p}{pp_s^*(\beta)}\right)
K(\mu,\beta)^{-\frac{p_s^*(\beta)}{p_s^*(\beta)-p}},
\left(\frac{p_s^*(\alpha)-p}{pp_s^*(\alpha)}\right)
K(\mu,\alpha)^{-\frac{p_s^*(\alpha)}{p_s^*(\alpha)-p}}\right\rbrace.\label{c*}
\end{align} 
\label{prop1}
\end{prop}

To prove Proposition~\ref{prop1} we need the following version of the mountain-pass theorem by Ambrosetti and Rabinowitz~\cite{MR0370183}. 
\begin{prop}
\label{teo3}
Let $(V,\Vert\, \cdot \, \Vert)$ be a Banach space and
consider a function $F\in C^1(V;\mathbb{R})$. 
Suppose that 
\begin{enumerate}
\item $F(0)=0$.
\item There exist $\lambda > 0$ and $R>0$ such that 
$F(u) \geqslant \lambda$ for all $u\in V$, 
with $\Vert u \Vert = R$.
\item There exist $v_0 \in V$ such that 
$\displaystyle\liminf_{t\rightarrow +\infty}F(tv_0)<0$.
\end{enumerate}
Let $t_0>0$ be a positive real number such that 
$\Vert t_0v_0\Vert > R$ and $F(t_0v_0)<0$. 
Define
\begin{align*}
c & \coloneqq \displaystyle\inf_{\gamma\in\Gamma}\sup_{t\in [0,1]}F(\gamma (t)) 
\end{align*}
where $ \Gamma \coloneqq \{ \gamma\in C([0,1],V) \colon 
\gamma (0)=0 \;\mbox{and}\;\gamma (1)=t_0v_0\} $. 
Then there exists a Palais-Smale $(PS)_c$ sequence at the level $c \in \mathbb{R}$ for the functional $F$.
\end{prop}

The proof of Proposition~\ref{prop1} follows from the next lemma.
\begin{lem}
\label{lem:functionalmpt}
The energy functional $\Phi$ verifies the hypotheses of 
Proposition~\ref{teo3} for every function 
$u\in D^{s,p}(\mathbb{R}^N)\setminus \{0\}$.\label{lema.hipotese.tpm}
\end{lem}
\begin{proof}
Clearly, $\Phi\in C^1(D^{s,p}(\mathbb{R}^N))$ and 
$\Phi (0)=0$. By definition~\eqref{6} of the best constant $1/K(\mu,\alpha)$, we have 
\begin{align*}
\Phi (u) 
& \geqslant  \displaystyle\frac{\Vert u\Vert^p}{p} -\frac{K(\mu,\beta)^{\frac{p_s^*(\beta)}{p}}\Vert u\Vert^{p_s^*(\beta)}}{p_s^*(\beta)} -\frac{K(\mu,\alpha)^{\frac{p_s^*(\alpha)}{p}}\Vert u\Vert^{p_s^*(\alpha)}}{p_s^*(\alpha)}\\
& =  \displaystyle\left( \frac{1}{p} -\frac{K(\mu,\beta)^{\frac{p_s^*(\beta)}{p}}\Vert u\Vert^{p_s^*(\beta)-p}}{p_s^*(\beta)} -\frac{K(\mu,\alpha)^{\frac{p_s^*(\alpha)}{p}}\Vert u\Vert^{p_s^*(\alpha)-p}}{p_s^*(\alpha)} \right)
\Vert u \Vert^p.
\end{align*}
Since $\alpha, \beta \in (0,sp)$, it follows that $p_s^*(\alpha)>p$ and $p^*_s(\beta)>p$. 
Therefore, from the pair of inequalities~\eqref{eq4fp} we deduce, for
$[u]_{s,p}=R$ suitably chosen, that there exists 
$\lambda>0 $ such that $\Phi (u) \geqslant \lambda > 0$.

Let $u\in D^{s,p}(\mathbb{R})$ be a nontrivial function; 
for $t>0$ it is valid the identity
\begin{align*}
\Phi (tu)=\frac{t^p\Vert u \Vert^p}{p} - \frac{t^{p^*_s(\beta)}}{p^*_s(\beta)}\displaystyle\int_{\mathbb{R}^N}\dfrac{|u|^{p^*_s(\beta)}}{\vert x\vert^{\beta}}dx -\frac{t^{p^*_s(\alpha)}}{p^*_s(\alpha)}\int_{\mathbb{R}^N}\frac{|u|^{p^*_s(\alpha)}}{\vert x \vert^{\alpha}}dx.
\end{align*}
This implies that
$\lim_{t\rightarrow +\infty}\Phi (tu)= -\infty$ 
as $t \rightarrow +\infty$.
So, we consider $t_u>0$ such that 
$\Phi (tu)<0$ for all $t \geqslant t_u$ and $[t_uu]_{s,p}>R$. Now we can define
\begin{align*}
\Gamma_u 
& \coloneqq 
\left\lbrace 
\gamma \in C([0,1],D^{s,p}(\mathbb{R}^N)) \colon 
\gamma (0)=0 \;\mbox{and}\;\gamma (1)=t_uu 
\right\rbrace
\end{align*}
and 
\begin{align*}
c_u \coloneqq \displaystyle\inf_{\gamma\in\Gamma_u}\sup_{t\in [0,1]}\Phi (\gamma (t)).
\end{align*}
It follows that the functional $\Phi$ verify the hypotheses of Proposition~\ref{teo3}. 
\end{proof}

Using Lemma~\ref{lem:functionalmpt} as well as Proposition~\ref{teo3}, we deduce the existence of a Palais-Smale sequence 
$\{u_n\}_n \subset D^{s,p}(\mathbb{R}^N) $ 
for the energy functional $\Phi$ 
such that 
\begin{align*}
\displaystyle\lim_{n\rightarrow + \infty}\Phi (u_n) =c_u 
\quad \mbox{and} \quad 
\displaystyle\lim_{n\rightarrow +\infty}\Phi' (u_n)=0
\quad\text{strongly in } (D^{s,p}(\mathbb{R}^N))'.
\end{align*}
Moreover, in the definition of $c_u$ we deduce also that
$c_u \geqslant \lambda > 0$; 
so, 
$ c_u > 0 $ for all $ u \in D^{s,p}(\mathbb{R}^N)\setminus\{0\}$.

\begin{lem}
Suppose that $\mu \in [0,\mu_H)$ and that 
$\alpha \in [0,sp)$. 
Then there exists 
$u \in D^{s,p}(\mathbb{R}^N)\setminus \{0\}$ 
such that $ 0 < c_u < c_*$,
where $c_*$ is defined in~\eqref{c*}.
\label{2.2a}
\end{lem}
\begin{proof}
By hypothesis on the parameters $\mu$ and $\alpha$, 
we can consider a function 
$u \in D^{s,p}(\mathbb{R}^N)\setminus \{0\}$ 
for which $1/K(\mu,\alpha)$ is attained; see Theorem~\ref{teo2}. By definition of $t_u$ and by the fact that $c_u>0$, we get
\begin{align*}
0 < c_u \leqslant \displaystyle\sup_{t\geqslant 0}\Phi (tu)\leqslant\displaystyle\sup_{t\geqslant 0}f(t), 
\end{align*}
where $f \colon \mathbb{R}_+ \to \mathbb{R}$ is defined by
\begin{align*}
f(t) & \coloneqq 
\dfrac{t^p\Vert u\Vert^p}{p}-\dfrac{t^{p_s^*(\alpha)}}{p_s^*(\alpha)}\displaystyle\int_{\mathbb{R}^N}\dfrac{|u|^{p_s^*(\beta)}}{\vert x \vert^{\alpha}}\,dx.
\end{align*}
Note that the supremum of the function $f$ is attained in 
$t_u \in \mathbb{R}_+$
such that
\begin{align*}
f(t_u) & = \dfrac{sp-\alpha}{p(N-\alpha)}
K(\mu ,\alpha)^{\frac{-(N-\alpha)}{sp-\beta}}.
\end{align*} 
We deduce that 
$0 < c_u \leqslant 
\left(\frac{1}{p}- \frac{1}{p^*_s(\alpha)}\right)
K(\mu,\alpha)^{\frac{-(N-\alpha)}{sp-\alpha}}$, 
since $u$ attains the constant $1/K(\mu,\alpha)$. 

Note that the inequality is strict, that is, the equality does not occur. Indeed, suppose that it is valid the equality; therefore, we have
\begin{align*}
0< c_u = \displaystyle\sup_{t\geqslant 0} \Phi (tu)
= \displaystyle\sup_{t\geqslant 0}f(t).
\end{align*}
Consider two positive real numbers $t_1,t_2 \in \mathbb{R}_+$ where two extrema are attained. Then
\begin{align*}
f(t_1) -\frac{t_1^{p_s^*(\beta)}}{p_s^*(\beta)}
\displaystyle\int_{\mathbb{R}^N}
\frac{|u|^{p_s^*(\beta)}}{|x|^{\beta}} \,dx 
= f(t_2).
\end{align*}
This implies that
$ f(t_1)>f(t_2)$
since $u\in D^{s,p}(\mathbb{R}^N)\setminus\{0\}$
and $t>0$. 
In this way we get a contradiction, 
for $\displaystyle\sup_{t\geqslant 0}f(t)$
is attained at $t_2>0$. 

Similarly, we get 
$c_u < \left(\frac{1}{p}- \frac{1}{p^*_s(\beta)}\right)
K(\mu,\beta)^{\frac{-(N-\beta)}{sp-\beta}}$.

The lemma is proved.
\end{proof}

\section{The structure of the Palais-Smale sequences}
\label{sec:structureps}
In this section we consider the parameter
$\alpha\in (0,sp)$. 
\begin{prop}
Let $\{u_n\}_n\subset D^{s,p}(\mathbb{R}^N)$ be a 
Palais-Smale sequence $(PS)_c$ for the energy functional $\Phi$ at the level $c\in (0,c_*)$, as defined in 
Proposition~\ref{prop1}.
Suppose that 
$ u_n \rightharpoonup 0$ weakly in $D^{s,p}(\mathbb{R}^N)$ 
as $n\rightarrow +\infty$. 
Then there exists a positive constant 
$\varepsilon_0 = \varepsilon_0(N,p,\mu,\alpha,s,c) $ 
such that for every $\delta > 0$ one of the following limits 
is valid, 
\begin{alignat}{2}
\displaystyle\lim_{n\rightarrow +\infty}
\int_{B_\delta(0)}
\frac{|u_n|^{p_s^*(\alpha)}}{\vert x\vert^{\alpha}}\,dx
& = 0
& \quad \text{or} \quad 
\displaystyle\lim_{n\rightarrow +\infty}
\int_{B_\delta(0)}
\frac{|u_n|^{p_s^*\alpha}}{\vert x\vert^{\alpha}}\,dx
& \geqslant \varepsilon_0.
\label{def.epsilon0}
\end{alignat}
\label{2}  
\end{prop}
The proof of Proposition~\ref{2} uses the three lemmas and one remark.
\begin{lem}
Let $\{u_n\}_n\subset D^{s,p}(\mathbb{R}^N)$ 
be a Palais-Smale sequence 
$(PS)_c$ for the energy functional $\Phi$ as defined in Proposição~\ref{2}. 
If $u_n\rightharpoonup 0$ weakly in $D^{s,p}(\mathbb{R}^N)$ 
as $n \rightarrow +\infty$, then for every compact subset 
$\omega \Subset \mathbb{R}^N\setminus\{0\}$, 
up to a subsequence, the following limits are valid,
\begin{alignat}{2}
&\displaystyle\lim_{n \rightarrow +\infty}
\int_{\omega}\frac{|u_n|^p}{|x|^{ps}} \,dx=
\lim_{n \rightarrow +\infty}\int_{\omega}\frac{|u_n|^{p_s^*(\alpha)}}{|x|^{\alpha}}\, dx 
 = 0,
\label{10}\\
&\displaystyle\lim_{n\rightarrow +\infty}
\int_{\omega}\dfrac{|u_n|^{p_s^*(\beta)}}{\vert x \vert^{\beta}}\,dx=
\lim_{n\rightarrow +\infty}\int_{\omega}
\int_{\omega}\dfrac{|u_n(x)-u_n(y)|^p}{|x-y|^{N+sp}}
\,dx\,dy
 = 0.\label{11}
\end{alignat}
\label{lema.ref.af.3.1}
\end{lem}
\begin{proof}
We consider a fixed compact subset 
$\omega\Subset\mathbb{R}^N\setminus\{0\}$. 
It is well known that the embedding 
$D^{s,p}(\mathbb{R}^N)\hookrightarrow L^q(\omega)$ is compact for the parameter $q$ in the interval 
$1 \leqslant q < p_s^*$,
where $p_s^*=\frac{Np}{N-sp}$ is the critical Sobolev exponent for the embedding; see Di Nezza, Palatucci and Valdinoci~\cite{MR2944369}. 
Note tha both
$|x|^{-sp}$ and $|x|^{-\alpha}$ are bounded in the subset 
$\omega$. Therefore, the limits in~\eqref{10} follow from the compact embedding, since by hypothesis 
$u_n\rightharpoonup 0$ weakly in $D^{s,p}(\mathbb{R}^N)$
as $n\rightarrow+\infty$ 
and we also have $ p < p_s^*(\alpha) < p_s^*$ 
due to the condition $\alpha\in (0,sp)$. 

Now we consider the limits in~\eqref{11}. 
Let $\displaystyle\eta \in C_0^{\infty}(\mathbb{R}^N;\mathbb{R})$ be a cut-off function such that
$\operatorname{supp}(\eta) 
\Subset \mathbb{R}^N \setminus\{0\}$ 
with $ 0 \leqslant \eta \leqslant 1$ and 
$\eta|_{\omega} \equiv 1$. 
Using a result in Brasco, Squassina and 
Yang~\cite[Lemma~A.1]{MR3732174} 
we deduce that $\eta^pu_n\in D^{s,p}(\mathbb{R}^N)$
$n \in \mathbb{N}$. 
So,
\begin{align}
\langle \Phi' (u_n),\eta^pu_n\rangle
& = o(\Vert \eta^p u_n\Vert )=o(\Vert u_n\Vert)=o_n(1)
\quad \text{as } n\rightarrow +\infty,\label{12}
\end{align}
since the sequence $\{u_n\}_n \subset D^{s,p}(\mathbb{R}^N)$
is bounded by the fact that
$\{u_n\}_n$ converges weakly to zero in 
$D^{s,p}(\mathbb{R}^N)$. 
In this way, from the estimates~\eqref{12} it follows that
\begin{align}
o_n(1) &=\langle \Phi' (u_n),\eta^pu_n\rangle \nonumber \\        
       &=\displaystyle\int_{\mathbb{R}^{N}}\int_{\mathbb{R}^{N}}
       \dfrac{J_{p}(u_n(x)-u_n(y))(\eta^p(x)u_n(x)-\eta^p(y)u_n(y))}{|x-y|^{N+sp}}\,dx\,dy \nonumber \\       
       &\qquad - \displaystyle\int_{\mathbb{R}^N}
       \dfrac{J_{p_s^*(\beta)}\eta^p u_n}
             {\vert x\vert^{\beta}}dx
       -\displaystyle\int_{\mathbb{R}^N}
       \dfrac{J_{p_s^*(\alpha)} \eta^p u_n}
             {|x|^{\alpha}}\,dx. 
\label{eqphi}
\end{align}
Note that both $|x|^{-\alpha}$ and $\vert x\vert^{-\beta}$
are bounded in  
$\operatorname{supp}(\eta) 
\Subset\mathbb{R}^N\setminus\{0\}$. 
Since $u_n\rightharpoonup 0$ weakly in 
$D^{s,p}(\mathbb{R}^N)$ as $n \rightarrow+\infty$
and since $D^{s,p}(\mathbb{R}^N)$ is compactly embedded in both 
$L^{p_s^*(\alpha)}_{\operatorname{loc}}(\mathbb{R}^N)$ and 
$L^{p_s^*(\beta)}_{\operatorname{loc}}(\mathbb{R}^N)$, we get 
\begin{align*}
\displaystyle\int_{\mathbb{R}^N}
\dfrac{J_{p_s^*(\alpha)}(u_n)\eta^pu_n}{|x|^{\alpha}}\,dx 
& = \displaystyle\int_{\mathbb{R}^N}\dfrac{|u_n|^{p_s^{*}(\alpha)}\eta^p}{|x|^{\alpha}}\,dx 
\rightarrow 0 \quad\mbox{as } n\rightarrow +\infty.
\shortintertext{and}
\displaystyle\int_{\mathbb{R}^N}
\dfrac{J_{p_s^*(\beta)}(u_n)\eta^pu_n}{|x|^{\beta}}\,dx 
& = \displaystyle\int_{\mathbb{R}^N}
\dfrac{|u_n|^{p_s^{*}(\beta)}\eta^p}{|x|^{\beta}}\,dx \rightarrow 0 \quad \mbox{as } n\rightarrow +\infty.
\end{align*}
Therefore, from the estimate~\eqref{eqphi} we obtain
\begin{align}
o_n(1) 
& = \displaystyle\int_{\mathbb{R}^{N}}\int_{\mathbb{R}^{N}}
\frac{J_p(u_n(x)-u_n(y))
\left(\eta^p(x)u_n(x)-\eta^p(y)u_n(y)\right)}{|x-y|^{N+sp}}\,dx\,dy \nonumber\\  
& = \displaystyle\int_{\mathbb{R}^N}\int_{\mathbb{R}^{N}}
    \frac{J_p(u_n(x)-u_n(y))
     \left(\eta^p(x)-\eta^p(y)\right) 
     u_n(y)}{|x-y|^{N+sp}}\,dx\,dy\nonumber \\
& \qquad 
+ \displaystyle\int_{\mathbb{R}^N}\int_{\mathbb{R}^{N}}
  \frac{\eta^p(x)|u_n(x)-u_n(y)|^p}{|x-y|^{N+sp}}\,dx\,dy. \label{seqps}
\end{align}
To proceed further,
we define the functions 
\begin{align*}
D^sf(y) \coloneqq \displaystyle\int_{\mathbb{R}^N}\frac{|f(x)-f(y)|^p}{|x-y|^{N+sp}}dx
\end{align*}
and also
\begin{align*}
I & \coloneqq \displaystyle\int_{\mathbb{R}^N}\int_{\mathbb{R}^{N}}
\frac{J_p(u_n(x)-u_n(y))
\left(\eta^p(x)-\eta^p(y)\right) u_n(y)}{|x-y|^{N+sp}}\,dx\,dy.
\end{align*}

Our goal now is to show that 
$I\rightarrow 0$ as $n\rightarrow+\infty$. 
To to this, we use the Hölder inequality and deduce that 
\begin{align}
I & \leqslant \displaystyle\int_{\mathbb{R}^{N}}\int_{\mathbb{R}^{N}}\frac{|u_n(x)-u_n(y)|^{p-1}|u_n(y)||\eta^p(x)-\eta^p(y)|}{|x-y|^{N+sp}}\,dx\,dy \nonumber \\
&\leqslant\displaystyle
[u_n]_{s,p}^{p-1}\left(\int_{\mathbb{R}^{N}}\int_{\mathbb{R}^{N}}\frac{|u_n(y)|^p|\eta^p(x)-\eta^p(y)|^p}{|x-y|^{N+sp}} \, dx \, dy\right)^{\frac{1}{p}} \nonumber \\
& \leqslant C\displaystyle\left(\int_{\mathbb{R}^N}|u_n(y)|^p|D^s\eta^p(y)|^p \,dy\right)^{\frac{1}{p}},
\label{estiI}
\end{align}
where to get the last inequality we used the fact that 
$u_n\rightharpoonup 0$ weakly in $D^{s,p}(\mathbb{R}^N)$ 
as $n\rightarrow+\infty$.
Therefore, the sequence $\{u_n\} \subset D^{s,p}(\mathbb{R}^N)$ is bounded. 

Now we show that
$D^s\eta^p(y) \in L^{\infty}(\mathbb{R}^N) $.
Indeed, for $y\in \mathbb{R}^N$ we have 
\begin{align*}
D^s\eta^p(y)
& = \displaystyle\int_{\mathbb{R}^N}
\frac{|\eta^p(x)-\eta^p(y)|^p}{|x-y|^{N+sp}}\,dx \\
& = \displaystyle\int_{\mathbb{R}^N\setminus B_r(y)}
    \frac{|\eta^p(x)-\eta^p(y)|^p}{|x-y|^{N+sp}}\,dx
   +\displaystyle\int_{B_r(y)}
   \frac{|\eta^p(x)-\eta^p(y)|^p}{|x-y|^{N+sp}}\,dx, 
\end{align*}
for every positive real number $r>0$. 

To estimate the first term we note that
since $\eta\in C^{\infty}_0(\mathbb{R}^N)$,
it follows that $\eta$ is a bounded function; 
thus, there exists a positive constant $C_1>0$ such that 
\begin{align*}
\displaystyle\int_{\mathbb{R}^N\setminus B_r(y)}\frac{|\eta^p(x)-\eta^p(y)|^p}{|x-y|^{N+sp}}\,dx 
& \leqslant C_1\displaystyle\int_{\mathbb{R}^N\setminus B_r(y)}\frac{1}{|x-y|^{N+sp}}\,dx \\
& \leqslant \tilde{C_1}\int_r^{\infty}t^{-1-sp}\,dt < + \infty.
\end{align*} 
To estimate the second term, we also use the fact that 
$\eta^p\in C^{\infty}_0(\mathbb{R}^N)$,
and the mean value inequality to deduce that
\begin{align*}
\displaystyle\int_{B_r(y)}\frac{|\eta^p(x)-\eta^p(y)|^p}{|x-y|^{N+sp}}dx 
& \leqslant \displaystyle\int_{B_r(y)}\frac{|\nabla(\eta^p)(x-y)|^p}{|x-y|^{N+sp}}\,dx\\
& \leqslant C\displaystyle\int_{B_r(y)}|x-y|^{p-(N+sp)}\,dx < +\infty,
\end{align*} 
where $C$ is an estimate for the growth of the gradient of 
$\eta^p$. 
Thus we can deduce from both estimates that
$D^s\eta^p\in L^{\infty}(\mathbb{R}^N)$. 
Returning to the analysis of inequality~\eqref{estiI}, we also have  
\begin{align}
\displaystyle\int_{\mathbb{R}^N}|u_n(y)|^p|D^s\eta^p(y)|\,dy \nonumber
& = \displaystyle\int_{B_R(0)}|u_n(y)|^p|D^s\eta^p(y)|\,dy \\
& \quad +\displaystyle\int_{\mathbb{R}^N\setminus B_R(0)}
    |u_n(y)|^p|D^s\eta^p(y)|\,dy,
    \label{estiDs}
\end{align}
for every fixed positive radius $R \in \mathbb{R}_+$ 
which will be defined below.

To estimate the first integral in~\eqref{estiDs} we recall that $u_n\rightharpoonup 0$ weakly in $D^{s,p}(\mathbb{R}^N)$; 
it follows that
$u_n \rightarrow 0$ strongly in $L^q_{\operatorname{loc}}(\mathbb{R}^N)$
for every $q\in[1,p_s^*)$, 
as $n \rightarrow +\infty$. 
As we have already seen, 
$D^s\eta^p\in L^{\infty}(\mathbb{R}^N)$;
moreover, $u_n\rightarrow 0$ in $L^{p}_{\operatorname{loc}}(\mathbb{R}^N)$
as $n\rightarrow+\infty$. 
Thus, it follows that for every positive real number 
$\varepsilon >0$ there exists $n_0\in\mathbb{N}$ such that
for $n\geqslant n_0$ it is valid the inequalities
\begin{align}
\displaystyle\int_{B_R(0)}|u_n(y)|^p|D^s\eta^p(y)|\,dy
\leqslant \Vert D^s\eta^p\Vert_{\infty}
\int_{B_R(0)}|u_n(y)|^p \,dy
< \frac{\varepsilon}{2}. \label{esti2}
\end{align}

Now we are going to estimate the second integral on the right hand side of the equality~\eqref{estiDs}. 
By using the Hölder inequality with exponents
$p^*_s/p$ and $N/sp$ we get
\begin{align*}
\int_{\mathbb{R}^N\setminus B_R(0)}|u_n(y)|^p|D^s\eta^p(y)|\,dy &\leqslant \Vert u_n\Vert^p_{L^{p_s^*}(\mathbb{R}^N)}\Vert D^s\eta^p\Vert_{L^{N/sp}(\mathbb{R}^N\setminus B_{R}(0))}\\
& \leqslant C \Vert D^s\eta^p\Vert_{L^{N/sp}(\mathbb{R}^N\setminus B_{R}(0))}.
\end{align*}
Because the sequence $\{u_n\}_n \subset D^{s,p}(\mathbb{R}^N)$ is bounded. By the definition of $D^s\eta^p$ we have
\begin{align*}
\displaystyle\int_{\mathbb{R}^N\setminus B_R(0)}|D^s\eta^p(y)|^{\frac{N}{sp}}\,dy
=\int_{\mathbb{R}^N\setminus B_R(0)}\left(\int_{\mathbb{R}^N}\frac{|\eta^p(x)-\eta^p(y)|^p}{|x-y|^{N+sp}}\,dx\right)^{\frac{N}{sp}}\,dy.
\end{align*}
By the fact that $\eta\in C^{\infty}_0(\mathbb{R}^N)$, there exists a positive real number $r>0$ such that $\operatorname{supp}(\eta)\subset B_r(0)$. So, we choose $R>r>0$ as a first condition on the constant $R$. In this way, we get
\begin{align*}
 & \displaystyle\int_{\mathbb{R}^N\setminus B_R(0)}\left(\int_{\mathbb{R}^N}\frac{|\eta^p(x)-\eta^p(y)|^p}{|x-y|^{N+sp}}\,dx\right)^{\frac{N}{sp}}\,dy \\
 & \quad = \displaystyle\int_{\mathbb{R}^N\setminus B_R(0)}\left(\int_{\mathbb{R}^N}\frac{|\eta^p(x)|^p}{|x-y|^{N+sp}}dx\right)^{\frac{N}{sp}}\,dy \\
& \quad = \displaystyle\int_{\mathbb{R}^N\setminus B_R(0)}\left(\int_{B_r(0)}\frac{|\eta^p(x)|^p}{|x-y|^{N+sp}}\,dx\right)^{\frac{N}{sp}}\,dy\\
& \quad \leqslant\displaystyle\int_{\mathbb{R}^N\setminus B_R(0)}\left(\int_{B_r(0)}\frac{|\eta^p(x)|^p}{(|y|-r)^{N+sp}}\,dx\right)^{\frac{N}{sp}}\,dy \\
& \quad \leqslant\displaystyle\left(\int_{B_r(0)}|\eta^p(x)|^pdx\right)^{\frac{N}{sp}}\int_{\mathbb{R}^N\setminus B_R(0)}\frac{1}{(|y|-r)^\frac{(N+sp)N}{sp}}\,dy \\
& \quad =\Vert \eta^p\Vert_p^{\frac{N}{s}}\int_{\mathbb{R}^N\setminus B_R(0)}\frac{1}{(|y|-r)^\frac{(N+sp)N}{sp}}\,dy \\ 
& \quad \leqslant C\displaystyle\int_{R}^{\infty}\int_{\partial B_t(0)}\frac{dS(y)}{(t-r)^{\frac{(N+sp)N}{sp}}}\,dt\\
& \quad \leq C\displaystyle\int_{R-r}^{\infty}\frac{\rho^{N-1}+r^{N-1}}{\rho^{\frac{(N+sp)N}{sp}}}\,d\rho .
\end{align*}
 
From this we deduce the existence of a positive constant $C=C(N,r)>0$ such that 
\begin{align*}
\displaystyle\int_{\mathbb{R}^N\setminus B_R(0)}|D^s\eta^p(y)|^{\frac{N}{sp}}\,dy 
& \leqslant 
C\left(\frac{1}{(R-r)^{\frac{N^2}{sp}}}+\frac{1}{(R-r)^{\frac{(N+sp)N}{sp}-1}}\right).
\end{align*}
Passing to the limit as $R\rightarrow+\infty$ we get 
\begin{align*}
\displaystyle\int_{\mathbb{R}^N\setminus B_R(0)}|D^s\eta^p(y)|^{\frac{N}{sp}}dy\rightarrow 0.
\end{align*} 
Moreover, for every positive real number 
$\varepsilon>0$ there exists $R(\varepsilon)>0$ big enough such that  
\begin{equation}
\displaystyle\int_{\mathbb{R}^N\setminus B_R(0)}|D^s\eta^p(y)|^{\frac{N}{sp}}\,dy < \frac{\varepsilon}{2}.
\label{estDs2}
\end{equation}
Substituting inequalities~\eqref{esti2} and~\eqref{estDs2} into equality~\eqref{estiDs}, for every 
$\varepsilon>0$ there exists $n_0\in \mathbb{N}$ and $R=R(\varepsilon)>0$ such that if $n>n_0$, then
\begin{align*}
\displaystyle\int_{\mathbb{R}^N}|u_n(y)|^p|D^s\eta^p(y)|\,dy 
& < \varepsilon,
\end{align*}
that is, 
\begin{equation}
\displaystyle\int_{\mathbb{R}^N}|u_n(y)|^p|D^s\eta^p(y)|dy \rightarrow 0 \quad \mbox{as } n\rightarrow +\infty.\label{convIparazero}
\end{equation}
Thus, by using inequality~\eqref{estiI} 
and the limit~\eqref{convIparazero} it follows that
\begin{align}
I & \coloneqq \displaystyle\int_{\mathbb{R}^N}\frac{J_p(u_n(x)-u_n(y))\left[\eta^p(x)-\eta^p(y)\right] u_n(y)}{|x-y|^{N+sp}}\,dx\,dy = o_n(1).
\label{I=o1}
\end{align}
Combining the estimates~\eqref{I=o1} 
and~\eqref{seqps}, we get
\begin{equation}
o_n(1)=\displaystyle\int_{\mathbb{R}^{N}}\int_{\mathbb{R}^{N}}\frac{\eta^p(x)|u_n(x)-u_n(y)|^p}{|x-y|^{N+sp}}\,dx\,dy.\label{Estruseqps}
\end{equation}

Our goal now is to show that
\begin{equation}
[\eta u_n]_{s,p}^p=\displaystyle\int_{\mathbb{R}^{N}}\int_{\mathbb{R}^{N}}
\frac{\eta^p(x)|u_n(x)-u_n(y)|^p}{|x-y|^{N+sp}}\,dx\,dy + o_n(1). \label{ordemseminorma}
\end{equation}
To accomplish this, we use the elementary inequality
\begin{align*}
||X+Y|^p-|X|^p|\leqslant C_p(|X|^{p-1}+|Y|^{p-1})|Y|\quad\text{for all } X, Y \in\mathbb{R}^N,
\end{align*}
valid for $p\geqslant 1$; here, $C_p>0$ is a constant that depends only on the exponent $p$. 
We use this inequality with the choices
\begin{alignat*}{2}
X & = \eta(x)(u_n(x)-u_n(y))
& \quad \text{and} \quad
Y & = u_n(y)(\eta(x)-\eta(y)).
\end{alignat*}
Thus,
\begin{align*}
&\displaystyle\int_{\mathbb{R}^{N}}\int_{\mathbb{R}^{N}}
\frac{\vert |\eta(x)u_n(x)-\eta(y)u_n(y)|^p
-|\eta(x)(u_n(x)-u_n(y))|^p\vert}
{\vert x-y\vert^{N+sp}}\,dx\,dy\\
& \quad\leqslant C_p 
\displaystyle\int_{\mathbb{R}^{N}}\int_{\mathbb{R}^{N}}
\frac{\vert \eta(x)(u_n(x)-u_n(y))\vert^{p-1}
\vert u_n(y)(\eta(x)-\eta(y))\vert}
{\vert x-y\vert^{N+sp}}\,dx\,dy\\
&\qquad + C_p 
\displaystyle \int_{\mathbb{R}^{N}}\int_{\mathbb{R}^{N}}
\frac{\vert u_n(y)(\eta(x)-\eta(y))\vert^p}
     {\vert x-y\vert^{N+sp}}\,dx\,dy.
\end{align*}
Using the Hölder inequality, the definition of $D^s\eta$ and the fact that $\eta\in C_0^{\infty}(\mathbb{R}^N)$, we obtain
\begin{align}
{} &
\displaystyle\int_{\mathbb{R}^{N}}\int_{\mathbb{R}^{N}}
\frac{\vert\vert\eta(x)u_n(x)
 -\eta(y)u_n(y)\vert^p
 -\vert\eta(x)(u_n(x)
 -u_n(y))\vert^p\vert}
 {\vert x-y\vert^{N+sp}}\,dx\,dy \nonumber \\
&\quad\leqslant C\displaystyle[u_n]_{s,p}^{p-1}\left(
\displaystyle\int_{\mathbb{R}^{N}}\int_{\mathbb{R}^{N}}
\frac{\vert u_n(y)(\eta(x)-\eta(y))\vert^p}
{\vert x-y\vert^{N+sp}}\,dx\,dy\right)^{1/p} \nonumber \\
& \qquad\quad+C_p 
\displaystyle\int_{\mathbb{R}^{N}}\int_{\mathbb{R}^{N}} 
  \frac{\vert u_n(y)(\eta(x)-\eta(y))\vert^p}
       {\vert x-y\vert^{N+sp}}\,dx\,dy.
\label{desigualdadeseminorma}
\end{align}
We remark that by the definition of $D^{s}\eta$,
\begin{align*}
\displaystyle\int_{\mathbb{R}^{N}}\int_{\mathbb{R}^{N}}
\frac{\vert u_n(y)(\eta(x)-\eta(y))\vert^p}{\vert x-y\vert^{N+sp}}\,dx\,dy =\int_{\mathbb{R}^{N}}\vert u_n(y)\vert^p \vert D^s\eta(y) \vert \,dy.
\end{align*}
Using the limit~\eqref{convIparazero} in the inequality~\eqref{desigualdadeseminorma}, we get
\begin{align*}
\displaystyle\int_{\mathbb{R}^{N}}\int_{\mathbb{R}^{N}}
\frac{\vert\vert\eta(x)u_n(x)-\eta(y)u_n(y)\vert^p
-\vert\eta(x)(u_n(x)-u_n(y))\vert^p\vert}
{\vert x-y\vert^{N+sp}}\,dx\,dy = o_n(1).
\end{align*}
This implies that the estimate~\eqref{ordemseminorma} is valid.
Following up, using inequalities~\eqref{eq4fp} and
the estimates~\eqref{I=o1} and~\eqref{ordemseminorma} 
in the estimate~\eqref{seqps}, we deduce that
\begin{align*}
o_n(1) 
& = \displaystyle\int_{\mathbb{R}^N}\int_{\mathbb{R}^N}
    \frac{\eta^p(x)|u_n(x)-u_n(y)|^p}{|x-y|^{N+sp}}\,dx\,dy\\
&  = [\eta u_n]_{s,p}^p + o_n(1)
  \geqslant \Vert \eta u_n\Vert^p + o_n(1).
\end{align*}
Thus, we get 
\begin{align*}
\Vert\eta u_n\Vert^p\rightarrow 0
\quad \mbox{as }
n\rightarrow+\infty.
\end{align*}
Since $\eta\equiv 1$ in $\omega\Subset\mathbb{R}^N\setminus\{0\}$, it follows that
\begin{align*}
\int_{\omega}\int_{\omega}\dfrac{\vert u_n(x)-u_n(y)\vert^p}{\vert x-y\vert^{N+sp}}\,dx\,dy
\rightarrow 0\quad \mbox{as } n\rightarrow+\infty.
\end{align*}
This establishes the limits~\eqref{11}. The lemma is proved.
\end{proof}

\begin{rem}
In the proof of Lemma~\ref{lema.ref.af.3.1} we showed that for every cut-off function 
$\eta\in C^{\infty}_0(\mathbb{R}^N)$ 
and for every sequence 
$\{u_n\}_n\subset D^{s,p}(\mathbb{R}^N)$ 
such that
$u_n\rightharpoonup 0$ weakly in $D^{s,p}(\mathbb{R}^N)$, 
it is valid 
\begin{align*}
[\eta u_n]_{s,p}^{p}&= \displaystyle\lim_{n\rightarrow +\infty}
\int_{\mathbb{R}^N}\int_{\mathbb{R}^N}
\dfrac{\eta^p(x) |u_n(x)-u_n(y)|^p}{|x-y|^{N+sp}}\,dx\,dy +o_n(1)
\end{align*}

Moreover, if we consider the cut-off function 
$\eta\in C^{\infty}_0(\mathbb{R}^N)$ 
as in the proof of Lemma~\ref{lema.ref.af.3.1} and the sequence
$\{u_n\}_n\subset D^{s,p}(\mathbb{R}^N)$ 
is a Palais-Smale sequence $(PS)_c$ at the level $c \in (0,c^*)$ as in Proposition~\ref{2}, then we get 
\begin{align*}
[\eta u_n]_{s,p}^{p}&= \displaystyle\lim_{n\rightarrow +\infty}\int_{\mathbb{R}^{N}}\int_{\mathbb{R}^{N}}\dfrac{\eta^p(x) |u_n(x)-u_n(y)|^p}{|x-y|^{N+sp}}dxdy +o_n(1).
\end{align*}
In this way, we deduce that 
\begin{align}
\displaystyle\lim_{n\rightarrow +\infty}\int_{\omega}\int_{\mathbb{R}^{N}}\dfrac{|u_n(x)-u_n(y)|^p}{|x-y|^{N+sp}}dxdy
& = 0,\label{w2convdaseminorma}\\
\displaystyle\lim_{n\rightarrow +\infty}\int_{\mathbb{R}^{N}}\int_{\omega}\dfrac{|u_n(x)-u_n(y)|^p}{|x-y|^{N+sp}}dxdy & = 0, \label{w1convdaseminorma}
\end{align}
for every subset $\omega \Subset \mathbb{R}^N\setminus\{0\}$.\label{obs1}
\end{rem}

Let $\delta \in \mathbb{R}_+$ be fixed; we define the following quantities 
\begin{align}
\gamma 
& \coloneqq \displaystyle{\varlimsup_{n\rightarrow +\infty}} \int_{B_\delta(0)} \int_{B_\delta(0)}
       \frac{\vert u_n(x)-u_n(y)\vert^{p}}
            {\vert x-y\vert^{N+sp}} dx dy 
     - \mu \int_{B_\delta(0)}
       \frac{\vert u_n\vert^{p}}
            {\vert x\vert^{\alpha}} dx \label{def.gamma} \\
\lambda 
& \coloneqq \displaystyle{\varlimsup_{n\rightarrow +\infty}}
  \int_{B_\delta(0)}
  \dfrac{\vert u_n\vert^{p_s^*(\alpha)}}
        {\vert x\vert^{\alpha}} \,dx, \label{def.lambda}\\
\xi
& \coloneqq \displaystyle{\varlimsup_{n\rightarrow +\infty}}
  \int_{B_\delta(0)}
  \frac{\vert u_n\vert^{p_s^*(\beta)}}
       {\vert x\vert^\beta} \,dx \label{def.xi}.
\end{align}

From Remark~\ref{obs1} and from Lemma~\ref{lema.ref.af.3.1}
we deduce that the above quantities are well defined and are independent of the particular choice of $\delta>0$.
Indeed, consider a cut-off function 
$\eta\in C_0^{\infty}(\mathbb{R}^N)$ 
such that 
$\eta|_{B_\delta(0)}\equiv 1$. 
It follows that there exist positive real numbers 
$R>r>0$ such that $\operatorname{supp}(\eta)\subset B_r(0)\subset B_R(0)$.
Note that
\begin{align}
[\eta u_n]^p_{s,p} 
& = \displaystyle\int_{\mathbb{R}^N}\int_{\mathbb{R}^N}
\frac{\vert \eta(x)u_n(x)-\eta(y)u_n(y)\vert^{p}}
     {\vert x-y\vert^{N+sp}}\,dx\,dy \nonumber \\ 
&= \displaystyle\int_{\mathbb{R}^N}
   \int_{\mathbb{R}^N\setminus B_R(0)}
   \frac{\vert \eta(x)u_n(x)-\eta(y)u_n(y)\vert^{p}}
        {\vert x-y\vert^{N+sp}}\,dx\,dy \nonumber \\ 
&\qquad+\displaystyle\int_{\mathbb{R}^N}\int_{B_R(0)}
   \frac{\vert\eta(x)u_n(x)-\eta(y)u_n(y)\vert^p}
        {\vert x-y\vert^{N+sp}}\,dx\,dy \nonumber \\ 
&= \displaystyle\int_{\mathbb{R}^N}
   \int_{\mathbb{R}^N\setminus B_R(0)}
   \frac{\vert\eta (y)u_n(y)\vert^p}
   {\vert x-y\vert^{N+sp}}\,dx\,dy \nonumber \\
&\qquad
+\displaystyle\int_{\mathbb{R}^N}\int_{B_R(0)}   
   \frac{\vert\eta(x)u_n(x)-\eta(y)u_n(y)\vert^p}
        {\vert x-y\vert^{N+sp}}\,dx\,dy.
\label{gamabemdef}
\end{align}

Now we estimate both integrals on the right-hand side of equality~\eqref{gamabemdef}.
Considering the first integral, we have
\begin{align*}
& \displaystyle\int_{\mathbb{R}^N}
  \int_{\mathbb{R}^N\setminus B_R(0)}
  \frac{\vert\eta (y)u_n(y)\vert^p}
       {\vert x-y\vert^{N+sp}}\,dx\,dy \\       
&\qquad  = \int_{\mathbb{R}^N}\vert\eta(y)u_n(y)\vert^p
   \left( \int_{\mathbb{R}^N\setminus B_R(0)}
   \frac{1}{\vert x-y\vert^{N+sp}}\,dx 
   \right) \,dy \\
&\qquad  = \int_{B_r(0)}\vert\eta(y)u_n(y)\vert^p
   \left( \int_{\mathbb{R}^N\setminus B_R(0)}
   \frac{1}{\vert x-y\vert^{N+sp}}\,dx 
   \right) \,dy\\
& \qquad  \leqslant \int_{B_r(0)}\vert\eta(y)u_n(y)\vert^p
   \left( \int_{\mathbb{R}^N\setminus B_R(0)}
   \frac{1}{\left(\vert x\vert-r\right)^{N+sp}}\,dx
   \right) \,dy\\
&\qquad =\left( \int_{B_r(0)}\vert\eta(y)u_n(y)\vert^p \,dy \right)   
  \left( \int_{\mathbb{R}^N\setminus B_R(0)}
         \frac{1}{\left(\vert x\vert-r\right)^{N+sp}}
         \,dx \right).
\end{align*}
Since $N \geqslant 2$ and $sp < N$ it follows that
$ N-1-(N+sp)=-1-sp<-1$ and $-(N+sp)<-1$; moreover, $R>r$.
Hence,
\begin{align*}
\int_{\mathbb{R}^N\setminus B_R(0)}\frac{1}{\left(\vert x\vert-r\right)^{N+sp}} \,dx & < \infty.
\end{align*}
Recall that 
$u_n\rightharpoonup 0$ weakly in
$D^{s,p}(\mathbb{R}^N)$ as $n \to +\infty$ 
and that the embedding 
$D^{s,p}(\mathbb{R}^N)\hookrightarrow L_{\operatorname{loc}}^{q}(\mathbb{R}^N)$ is compact for $q\in [1,p_s^*)$;
so, 
\begin{align*}
\int_{B_r(0)}\vert\eta(y)u_n(y)\vert^p \,dy
& \rightarrow 0 \quad \textup{as } n\rightarrow \infty.
\end{align*}
From these results we deduce that
\begin{align*}
\displaystyle\int_{\mathbb{R}^N}
   \int_{\mathbb{R}^N\setminus B_R(0)}
   \frac{\vert\eta (y)u_n(y)\vert^p}
        {\vert x-y\vert^{N+sp}} \,dx\,dy
& \rightarrow 0 \quad \textup{as } n\rightarrow \infty.
\end{align*}
Thus, from equality~\eqref{gamabemdef} we obtain
\begin{align}
[\eta u_n]^p_{s,p}
& = \displaystyle\int_{\mathbb{R}^N}\int_{B_R(0)}
    \frac{\vert\eta(x)u_n(x)-\eta(y)u_n(y)\vert^p}
         {\vert x-y\vert^{N+sp}}\,dx\,dy + o_n(1).\label{ordemnormaetaun}
\end{align}

Considering the second integral in~\eqref{gamabemdef}, we also have
\begin{align}
{} & \displaystyle\int_{\mathbb{R}^N}\int_{B_R(0)}
   \frac{\vert\eta(x)u_n(x)-\eta(y)u_n(y)\vert^p}
        {\vert x-y\vert^{N+sp}} \,dx\,dy \nonumber \\
&\quad = \displaystyle\int_{ \mathbb{R}^N\setminus B_R(0)}
   \int_{B_R(0)}
   \frac{\vert\eta(x)u_n(x)-\eta(y)u_n(y)\vert^p}
        {\vert x-y\vert^{N+sp}} \,dx\,dy \nonumber \\
& \qquad \quad +\displaystyle\int_{B_R(0)}\int_{B_R(0)}
   \frac{\vert\eta(x)u_n(x)-\eta(y)u_n(y)\vert^p}
        {\vert x-y\vert^{N+sp}} \,dx\,dy.
\label{estetaun}
\end{align}
Using the facts that
$\operatorname{supp}(\eta)\subset B_r(0) \subset B_R(0)$ 
and $R>r>0$, we deduce that
\begin{align*}
{} & \displaystyle\int_{\mathbb{R}^N\setminus B_R(0)}
   \int_{B_R(0)}
   \frac{\vert\eta(x)u_n(x)-\eta(y)u_n(y)\vert^p}
        {\vert x-y\vert^{N+sp}} \,dx\,dy \\
& \quad =\displaystyle\int_{\mathbb{R}^N\setminus B_R(0)}
   \int_{B_R(0)}
   \frac{\vert\eta(x)u_n(x)\vert^p}
        {\vert x-y\vert^{N+sp}} \,dx\,dy \\
& \quad =\int_{\mathbb{R}^N\setminus B_R(0)}\int_{B_r(0)}
   \frac{\vert\eta(x)u_n(x)\vert^p}
        {\vert x-y\vert^{N+sp}} \,dx\,dy \\
& \quad \leqslant \displaystyle
  \left( \int_{\mathbb{R}^N\setminus B_R(0)}
         \frac{1}{\left(\vert y\vert-r\right)^{N+sp}} \,dy
  \right)
  \left( \int_{B_r(0)}\vert\eta(x)u_n(x)\vert^p \,dx \right).
\end{align*} 
Under the conditions on the parameters $N \geqslant$ 
and $sp < N$, we know that
\begin{align*}
\displaystyle \int_{\mathbb{R}^N\setminus B_R(0)}
  \frac{1}{\left(\vert y\vert-r\right)^{N+sp}} \,dy
  & < +\infty.
\end{align*} 
Using again the facts that 
$u_n\rightharpoonup 0$ weakly in 
$D^{s,p}(\mathbb{R}^N)$ 
and that the embedding 
$D^{s,p}(\mathbb{R}^N)\hookrightarrow L^{q}_{loc}(\mathbb{R}^N)$ is compact for $q\in [1,p_s^*)$,
it follows that   
\begin{align}
\lim_{n\rightarrow +\infty}\int_{\mathbb{R}^N\setminus B_R(0)}
   \int_{B_R(0)}
   \frac{\vert\eta(x)u_n(x)-\eta(y)u_n(y)\vert^p}
   {\vert x-y\vert^{N+sp}} \,dx\,dy=0.\label{estordemetaun}
\end{align}

Using the estimates~\eqref{ordemnormaetaun},~\eqref{estetaun} and~\eqref{estordemetaun}, it follows that
\begin{align}
[\eta u_n]_{s,p}^p & =\displaystyle\int_{B_R(0)}\int_{B_R(0)}\frac{\vert\eta(x)u_n(x)-\eta(y)u_n(y)\vert^p}{\vert x-y\vert^{N+sp}}dxdy +o_n(1).
\label{estordemestaun2}
\end{align}

To proceed further, we must estimate the previous integral.
To do this, we consider positive real numbers $R>\delta>0$
and we write
\begin{align*}
B_R(0)\times B_R(0)
 = \left( \left( B_R(0)\setminus B_\delta(0)\right) \cup B_\delta(0)\right) \times \left(\left( B_R(0)\setminus B_\delta(0)\right)\cup B_\delta(0)\right).
\end{align*}

In the case of the domain of integration 
$\left( B_R(0)\setminus B_\delta(0)\right) \times B_R(0)$
we have 
\begin{align}
{} & \displaystyle\int_{B_R(0)\setminus B_\delta(0)}
   \int_{B_R(0)}
   \frac{\vert\eta(x)u_n(x)-\eta(y)u_n(y)\vert^p}
        {\vert x-y\vert^{N+sp}} \,dx\,dy \nonumber \\
& \quad \leqslant 2^p
   \displaystyle\int_{B_R(0)\setminus B_\delta(0)}
   \int_{B_R(0)}
   \dfrac{\vert\eta(x)(u_n(x)-u_n(y))\vert^p}
        {\vert x-y\vert^{N+sp}} \,dx\,dy \nonumber\\
&  \qquad      + 2^p \displaystyle\int_{B_R(0)\setminus B_\delta(0)}
     \int_{B_R(0)}
     \dfrac{\vert u_n(y)(\eta(x)-\eta(y))\vert^p}
          {\vert x-y\vert^{N+sp}} \,dx\,dy .
\label{ineq:BRdelta}
\end{align}

To estimate the first integral on the right hand side of inequality~\eqref{ineq:BRdelta}
we use limit~\eqref{w2convdaseminorma} 
with $ \omega = B_R(0)\setminus B_\delta(0)\Subset\mathbb{R}^N\setminus\{0\}$
to deduce that
\begin{align*}
\lim_{n\rightarrow +\infty}\int_{B_R(0)\setminus B_\delta(0)}
\int_{B_R(0)}
\frac{\vert\eta(x)(u_n(x)-u_n(y))\vert^p}
     {\vert x-y\vert^{N+sp}} \,dx\,dy =0.
\end{align*}
To estimate the second integral on the right hand side of inequality~\eqref{ineq:BRdelta}, 
we use the same subset 
$ \omega = B_R(0)\setminus B_\delta(0)\Subset\mathbb{R}^N\setminus\{0\}$
together with limit~\eqref{convIparazero} to obtain  
\begin{align*}
\lim_{n\rightarrow +\infty}\int_{B_R(0)\setminus B_\delta(0)}
\int_{B_R(0)}
\frac{\vert u_n(y)(\eta(x)-\eta(y))\vert^p}
     {\vert x-y\vert^{N+sp}} \,dx\,dy=0.
\end{align*}
Therefore, from the two previous limits and from inequality~\eqref{ineq:BRdelta} we obtain 
\begin{align}
\displaystyle\lim_{n\rightarrow +\infty}\int_{B_R(0)\setminus B_\delta(0)}
\int_{B_R(0)}
\frac{\vert\eta(x)u_n(x)-\eta(y)u_n(y)\vert^p}
     {\vert x-y\vert^{N+sp}} \,dx\,dy=0.
\label{limit:BRdelta1}
\end{align}
Similarly, in the case of the domain of integration 
$B_\delta(0)\times\left(B_R(0)\setminus B_\delta(0)\right)$  
we proceed as in the previous case to obtain. We write an inequality analogous to inequality~\eqref{ineq:BRdelta}; 
afterwards, we use inequality~\eqref{w1convdaseminorma} together with limit~\eqref{convIparazero} to obtain
\begin{align}
\lim_{n\rightarrow +\infty}\int_{B_\delta(0)}
\int_{B_R(0)\setminus B_\delta(0)}
\frac{\vert\eta(x)u_n(x)-\eta(y)u_n(y)\vert^p}
     {\vert x-y\vert^{N+sp}} \,dx\,dy=0.
\label{limit:BRdelta2}
\end{align}

Therefore, combining limits~\eqref{limit:BRdelta1} 
and~\eqref{limit:BRdelta2} with 
estimate~\eqref{estordemestaun2},
we deduce that
\begin{align}
[\eta u_n]_{s,p}^p
& = \displaystyle\int_{B_R(0)}
    \int_{B_R(0)}
    \frac{\vert\eta(x)u_n(x)-\eta(y)u_n(y)\vert^p}
         {\vert x-y\vert^{N+sp}} \,dx\,dy 
    + o_n(1) \nonumber \\
& = \displaystyle\int_{B_\delta(0)}
    \int_{B_\delta(0)}
    \frac{\vert\eta(x)u_n(x)-\eta(y)u_n(y)\vert^p}
         {\vert x-y\vert^{N+sp}} \,dx\,dy 
    + o_n(1).
\label{est.eta.unfinal}
\end{align}

Now we can state the following result.

\begin{lem}
\label{lem:desig.envolvendo.lambda}
Let $\{u_n\}_n\subset D^{s,p}(\mathbb{R}^N)$ be a 
Palais-Smale sequence $(PS)_c$ for the functional $\Phi$ at the level 
$c \in (0,c^*)$ and let $\lambda$, $\xi$, and $\gamma$ 
be defined in~\eqref{def.gamma},~\eqref{def.lambda}
and~\eqref{def.xi}. 
If $u_n\rightharpoonup 0$ weakly in $D^{s,p}(\mathbb{R}^N)$ 
as $n \rightarrow +\infty$ then 
\begin{alignat}{2}
\lambda^{\frac{p}{p_s^*}}
& \leqslant K(\mu,\alpha)\gamma
& \quad \text{and} \quad
\xi^{\frac{p}{p_s^*(\beta)}}
& \leqslant K(\mu,\beta)\gamma \label{desig.envolvendo.lambda}
\end{alignat}
\end{lem}
\begin{proof}
Let $\eta\in C_0^{\infty}(\mathbb{R}^N)$ be a cut-off function such that $\eta|_{B_\delta(0)}\equiv 1$, 
with $\delta >0$. Using the definition~\eqref{6} of the constant $K(\mu,\alpha)$ we get
\begin{align*}
{} & \left(\int_{B_\delta(0)}
\dfrac{\vert u_n\vert^{p_s^*(\alpha)}}
      {\vert x\vert^{\alpha}} \,dx 
\right)^{\frac{p}{p_s^*(\alpha)}} \leqslant \left( \int_{\mathbb{R}^N}
  \dfrac{\vert\eta u_n\vert^{p_s^*(\alpha)}}
        {\vert x\vert^{\alpha}} \,dx
\right)^{\frac{p}{p_s^*(\alpha)}} \leqslant K(\mu,\alpha)\Vert\eta u_n\Vert^p \\
\end{align*} 
Using the estimate~\eqref{est.eta.unfinal} and that $u_n\rightharpoonup 0 $ in $D^{s,p}(\mathbb{R}^N)$, we concluded taking the limit as $n\rightarrow +\infty$ that
$ \lambda^{\frac{p}{p_s^*}} \leqslant K(\mu,\alpha)\gamma$.
The other inequality in~\eqref{desig.envolvendo.lambda} can be obtained in a similar way. This concludes the proof of the lemma.
\end{proof}

Now we state another lemma that will be useful in the proof of Proposition~\ref{2}.

\begin{lem}
\label{lem:desig:gammalambdaxi}
Let $\{u_n\}_n\subset D^{s,p}(\mathbb{R}^N)$ be a 
Palais-Smale sequence $(PS)_c$ for the functional $\Phi$ at the level 
$c \in (0,c^*)$ and let $\gamma$, $\lambda$, and $\xi$ be 
defined in~\eqref{def.gamma},~\eqref{def.lambda}, 
and~\eqref{def.xi}. 
If $u_n\rightharpoonup 0$ weakly in $D^{s,p}(\mathbb{R}^N)$ 
as $n\rightarrow+\infty$, 
then $\gamma\leqslant\lambda+\xi$. 
\end{lem}
\begin{proof}
Let $\eta\in C^{\infty}_0(\mathbb{R}^N)$ 
be a cut-off function such that
$\eta|_{B_\delta(0)}\equiv 1$, with $\delta>0$. 
Using Brasco, Squassina and Yang~\cite[Lemma A.1]{MR3732174}, it follows that 
$\eta^p u_n\in D^{s,p}(\mathbb{R}^N)$;
hence, 
\begin{align*}
\lim_{n\rightarrow+\infty}\langle\Phi'(u_n),\eta^pu_n\rangle & = 0,
\end{align*}
that is,
\begin{align*}
o_n(1)
& = \displaystyle\int_{\mathbb{R}^{N}}\int_{\mathbb{R}^{N}}
    \frac{J_p(\eta u_n(x)-\eta u_n(y))
         \left( \eta^p(x)u_n(x)-\eta^p(y)u_n(y)\right)}
         {\vert x-y\vert^{N+sp}} \,dx\,dy \\
& \qquad - \mu \int_{\mathbb{R}^N}
  \frac{\vert \eta u_n\vert^p}
       {\vert x\vert^{ps}} \,dx  
  - \int_{\mathbb{R}^N}
    \dfrac{\vert u_n\vert^{p_s^*(\beta)}\eta^p}
          {\vert x\vert^{\beta}} \,dx
  - \int_{\mathbb{R}^N}
    \dfrac{\vert u_n\vert^{p_s^*(\alpha)}\eta^p}
          {\vert x\vert^{\alpha}} \,dx.
\end{align*}
We know that 
\begin{align*}
 & \displaystyle\int_{\mathbb{R}^{N}}\int_{\mathbb{R}^{N}}
    \frac{J_p(\eta u_n(x)-\eta u_n(y))(\eta^p(x)u_n(x)-\eta^p(y)u_n(y))}
         {\vert x-y\vert^{N+sp}} \,dx\,dy = \\
&\qquad\qquad = [\eta u_n]^p_{s,p} + o_n(1) = \displaystyle\int_{B_\delta(0)}\int_{B_\delta(0)}
    \frac{\vert u_n(x)-u_n(y)\vert^p}
         {\vert x-y\vert^{N+sp}} \,dx\,dy + o_n(1).
\end{align*} 
Therefore, by $\limsup$ property and Lemma~\ref{lema.ref.af.3.1} with $\omega =\mbox{supp}(\eta)\setminus B_\delta(0)\Subset\mathbb{R}^N\setminus\{0\}$ we have 
$\gamma\leqslant\lambda+\xi$.
This concludes the proof of the lemma.
\end{proof}

Finally, we can prove Proposition~\ref{2}, which states that every Palais-Smale 
$\{u_n\}_n\subset D^{s,p}(\mathbb{R}^N)$ for the functional $\Phi$ at the level $c\in (0,c_*)$ such that 
$ u_n \rightharpoonup 0$ weakly in $D^{s,p}(\mathbb{R}^N)$ 
as $n\rightarrow +\infty$ verifies one of the limits
$\lim_{n\rightarrow +\infty}
\int_{B_\delta(0)}
\frac{|u_n|^{p_s^*(\alpha)}}{\vert x\vert^{\alpha}}\,dx
= 0$
or 
$\lim_{n\rightarrow +\infty}
\int_{B_\delta(0)}
\frac{|u_n|^{p_s^*\alpha}}{\vert x\vert^{\alpha}}\,dx
 \geqslant \varepsilon_0$
with arbitrary $\delta > 0$
and a positive constant 
$\varepsilon_0 = \varepsilon_0(N,p,\mu,\alpha,s,c) $ 
independent of $\delta$.

\begin{proof}[Proof of Proposition~\ref{2}]
Let $\{u_n\}_n\ \subset D^{s,p}(\mathbb{R}^N)$ be a Palais-Smale sequence for the functional $\Phi$ at the level $c \in (0,c^*)$ as in Proposição~\ref{prop1} and consider
$\alpha \neq 0$. From Lemmas~\ref{lem:desig.envolvendo.lambda} and~\ref{lem:desig:gammalambdaxi}, we infer that
\begin{align*}
\lambda^{\frac{p}{p_s^*}}
& \leqslant K(\mu,\alpha)\gamma\leqslant K(\mu,\alpha)\lambda+K(\mu,\alpha)\xi.
\end{align*}
So,
\begin{equation}
\lambda^{\frac{p}{p_s^*(\alpha)}}\left(1-K(\mu,\alpha)\lambda^{1-\frac{p}{p_s^*(\alpha)}}\right)\leqslant K(\mu,\alpha)\xi.\label{des.lambda.beta}
\end{equation}
Since 
\begin{align*}
\Phi(u_n) - \left\langle \dfrac{1}{p_s^*(\alpha)}\Phi'(u_n),u_n\right\rangle
& = c + o_n(1),
\end{align*}
it follows from Lemma~\ref{lema.ref.af.3.1} that
\begin{align}
\lambda\leqslant \frac{cp(N-\alpha)}{sp-\alpha}\label{desg.lambda.c.N.s}.
\end{align}
Combining inequalities~\eqref{des.lambda.beta}
and~\eqref{desg.lambda.c.N.s}, we deduce that 
\begin{align*}
\left( 1
 -\left(\frac{cp(N-\alpha)}{sp-\alpha}
  \right)^{\frac{sp-\alpha}{p(N-\alpha)}}
  K(\mu,\alpha)\right) 
  \lambda^{\frac{p}{p_s^*(\alpha)}}
& \leqslant K(\mu,\alpha)\xi.
\end{align*}
By the definition~\eqref{c*} of $c^*$, 
it follows that 
$ 1-\left(\frac{cp(N-\alpha)}{(sp-\alpha)}
  \right)^{\frac{sp-\alpha}{p(N-\alpha)}}
  K(\mu,\alpha) > 0$. 
So, there exists a positive real number 
$\delta_1= \delta_1(N,p,\mu,s,\alpha,c)>0$
such that 
$\lambda^{\frac{p}{p_s^*}}\leqslant \delta_1\xi$.

Similarly, we deduce the existence of another positive real number $\delta_2 = \delta_2(N,p,\mu,s,\beta,c)>0$ such that 
$\xi^{\frac{p}{p_s^*(\alpha)}}\leqslant\delta_2\lambda$.

In this way, we infer that 
$\lambda = 0$ if, and only if, $\xi = 0$ and 
$\lambda > 0$ if, and only if, $\xi > 0$. 
Thus, by definition~\ref{def.lambda}, 
we can guarantee the existence of a real positive number 
$\varepsilon_0 = \varepsilon_0(N,p,\mu,\alpha,s,c)>0$ 
such that for all $\delta > 0$ one of the limits is valid,
\begin{alignat*}{2}
\displaystyle\lim_{n\rightarrow +\infty}
\int_{B_\delta(0)}
\frac{|u_n|^{p_s^*(\alpha)}}{\vert x\vert^{\alpha}}\,dx
& = 0
& \quad \text{or} \quad 
\displaystyle\lim_{n\rightarrow +\infty}
\int_{B_\delta(0)}
\frac{|u_n|^{p_s^*\alpha}}{\vert x\vert^{\alpha}}\,dx
& \geqslant \varepsilon_0.
\end{alignat*}
The proposition is proved.
\end{proof}

\section{Proof of Theorem~\ref{teo1}}
\label{sec:proofteo1}
The proof of Theorem~\ref{teo1} uses three auxiliary lemmas.
\begin{lem}
\label{lem:limsuppositive}
Let $\{u_n\}_n\subset D^{s,p}(\mathbb{R}^N)$ be a Palais-Smale sequence $(PS)_c$ for the functional $\Phi$ at the level $c \in (0,c^*)$ as in Proposition~\ref{2}. Therefore, 
\begin{align*}
\varlimsup_{n\rightarrow+\infty}
\int_{\mathbb{R}^N}
\frac{\vert u_n\vert^{p_s^*(\alpha)}}
     {\vert x\vert^{\alpha}} \,dx 
& > 0.
\end{align*}
\end{lem}
\begin{proof}
We argue by contradiction. Suppose that
\begin{equation}
\lim_{n\rightarrow+\infty}\int_{\mathbb{R}^N}\frac{\vert u_n\vert^{p_s^*(\alpha)}}{\vert x\vert^{\alpha}}dx=0.\label{cond.contrad.}
\end{equation}
Using this hypothesis and the fact that
$\langle\Phi'(u_n),u_n\rangle =o_n(1)$, it follows that
\begin{align*}
\Vert u_n\Vert^{p}
& = \Vert u_n\Vert^{p_s^*(\beta)}_{L^{p_s^*(\beta)}(\mathbb{R}^N;\vert x\vert^{-\beta})}+o_n(1).
\end{align*}
From this estimate and by the definition~\eqref{6} of the constant 
$K(\mu,\beta)$, it follows that
\begin{align*}
\Vert u_n\Vert^{p}_{L^{p_s^*(\beta)}(\mathbb{R}^N;\vert x\vert^{-\beta})}
& \leqslant K(\mu,\beta)\Vert u_n\Vert^p \\
& = K(\mu,\beta)\left(\Vert u_n\Vert^{p_s^*(\beta)}_{L^{p_s^*(\beta)}(\mathbb{R}^N;\vert x\vert^{-\beta})}+o_n(1)\right),
\end{align*}
that is,
\begin{align*}
\Vert u_n\Vert^{p}_{L^{p_s^*(\beta)}(\mathbb{R}^N;\vert x\vert^{-\beta})} 
\left( 1-K(\mu,\beta)\Vert u_n\Vert^{p_s^*(\beta)-p}_{L^{p_s^*(\beta)}(\mathbb{R}^N;\vert x\vert^{-\beta})}\right)
& \leqslant o_n(1).
\end{align*}
We already know that 
\begin{align*}
\Phi(u_n)-\dfrac{1}{p}\langle\Phi'(u_n),u_n\rangle
& = c + o_n(1);
\end{align*} 
thus, using the hypothesis~\eqref{cond.contrad.} again, 
we get
\begin{align}
\Vert u_n \Vert_{L^{p_{\beta}^{*}}}^{p_{\beta}^{*}} 
& = \int_{\mathbb{R}^N}
\frac{\vert u_n\vert^{p_s^*(\beta)}}
     {\vert x\vert^{\beta}} \,dx
 = \frac{cp(N-\beta)}{(ps-\beta)} + o_n(1) \nrightarrow 0
\quad \text{as } n\rightarrow+\infty.
\label{eq:estimatenormbeta}
\end{align}
These estimates mean that
\begin{align*}
\Vert u_n\Vert^{p}_{L^{p_s^*(\beta)}(\mathbb{R}^N;\vert x\vert^{-\beta})}
\left( 1-K(\mu,\beta)\left(\frac{cp(N-\beta)}{(ps-\beta)}\right)^{\frac{p_s^*(\beta)-p}{p_s^*(\beta)}}
\right)
& \leqslant o_n(1).
\end{align*}
By the definition~\eqref{c*} of the constant $c^*$ and the hypothesis $c \in (0,c^*)$, it follows that
\begin{align*}
\lim_{n\rightarrow+\infty}\Vert u_n\Vert^p_{L^{p_s^*(\beta)}(\mathbb{R}^N,\vert x\vert^{-\beta})}
& = 0.
\end{align*}
But this is a contradiction with inequality~\eqref{eq:estimatenormbeta}.
The lemma is proved.
\end{proof}

\begin{lem}
Let $\{u_n\}_n \subset D^{s,p}(\mathbb{R}^{N})$ be a Palais-Smale sequence $(PS)_c$ for the functional $\Phi$ at the level $c \in (0,c^*)$. Then there exists a positive real number 
$\varepsilon_1 \in (0,\frac{\varepsilon_0}{2}]$, 
with $\varepsilon_0$ given in the limit~\eqref{def.epsilon0}, 
such that for every positive real number 
$\varepsilon\in (0,\varepsilon_1)$, 
there exists a sequence of positive real numbers
$\{r_n\}_n \subset\mathbb{R}_+$ with the property that the sequence
$\{\tilde{u}_n\}_n\subset D^{s,p}(\mathbb{R}^N)$, 
defined by the conformal transformation
\begin{align}
\tilde{u}_n(x) \coloneqq r_n^{\frac{N-sp}{p}}u_n(r_nx)
\quad \text{for all } x\in\mathbb{R}^N,\label{def.u.tilde}
\end{align}
is another $(PS)_c$ sequence which verifies
\begin{align}
\int_{B_1(0)}\frac{\vert \tilde{u}_n\vert^{p_s^*(\alpha)}}{\vert x\vert^{\alpha}} \,dx
& = \varepsilon \quad \text{for all } n\in\mathbb{N}.\label{norma.seq.u.tilde=epsilon}
\end{align}
\end{lem}
\begin{proof}
We begin by setting
\begin{align*}
\sigma & \coloneqq \displaystyle
\varlimsup_{n\rightarrow+\infty}
\int_{\mathbb{R}^N}
\frac{\vert u_n\vert^{p_s^*(\alpha)}}
     {\vert x \vert^{\alpha}} \,dx.
\end{align*}
From Lemma~\ref{lem:limsuppositive} it follows that
$\sigma>0$. 
Let 
$\varepsilon_1 \coloneqq \min\{\frac{\varepsilon_0}{2},\sigma\}$, with $\varepsilon_0>0$ given in the limit~\eqref{def.epsilon0}
and let us fix $\varepsilon\in (0,\varepsilon_1)$.
Passing to a subsequence if necessary, still denoted in the same way,
for every natural number $n\in\mathbb{N}$ 
there exists a positive real number $r_n>0$ such that 
\begin{align*}
\int_{B_{r_n}(0)}
\frac{\vert u_n\vert^{p_s^*(\alpha)}}
     {\vert x\vert^{\alpha}} \,dx
& = \varepsilon.
\end{align*}
The change of variables
$x = r_n y$ yields
\begin{align*}
 \varepsilon=\int_{B_{r_n}(0)}\frac{\vert u_n(x)\vert^{p_s^*(\alpha)}}{\vert x\vert^{\alpha}} \,dx
& = \int_{B_1(0)}
    \frac{\vert u_n(r_ny)\vert^{p_s^*(\alpha)}r_n^N}
         {\vert y\vert^{\alpha}r_n^{\alpha}} \,dy\\
& = \int_{B_1(0)}
    \frac{\vert r_n^{\frac{N-\alpha}{p_s^*(\alpha)}}
     u_n(r_ny)\vert^{p_s^*(\alpha)}}
     {\vert y\vert^{\alpha}} \,dy \\
& = \int_{B_1(0)}
 \frac{\vert \tilde{u}_n\vert^{p_s^*}}{\vert y \vert^{\alpha}} \,dy,
\end{align*} 
A similar change of variables also yields 
\begin{align*}
[\tilde{u}_n]_{s,p}^p
& = \int_{\mathbb{R}^{N}}\int_{\mathbb{R}^{N}}
    \frac{\vert r_n^{\frac{N-ps}{p}}u_n(r_nx)
         -r_n^{\frac{N-ps}{p}}u_n(r_ny)\vert^{p}}
         {|x-y|^{N+sp}} \,dx\,dy \\
& = r_n^{N-ps}\int_{\mathbb{R}^{N}}\int_{\mathbb{R}^{N}}
    \frac{\vert u_n(z)-u_n(t)\vert^p}
         {r_n^{-(N+sp)}|z-t|^{N+sp}} r_n^{-2N} \,dz\,dt \\
& = [u_n]_{s,p}^p.
\end{align*}
Moreover,
\begin{align*}
\int_{\mathbb{R}^N}
\frac{|\tilde{u}_n|^p}{|x|^{ps}} \,dx
& = r_n^{N-ps} \int_{\mathbb{R}^N}
    \frac{|u_n(y)|^p}{r_n^{-ps}|y|^{ps}}r_n^{-N} \,dy
  = \int_{\mathbb{R}^N}
    \frac{|u_n|^p}{|y|^{ps}} \,dy,\\
\shortintertext{and also}
\int_{\mathbb{R}^N}
\frac{|\tilde{u}_n|^{p_s^*(\beta)}}{|x|^{\beta}} \,dx
& = r_n^{N-\beta}\int_{\mathbb{R}^N}
    \frac{|u_n(y)|^{p_s^*(\beta)}}
         {r_n^{-\beta}|y|^{\beta}} r_n^{-N} \,dy
  = \int_{\mathbb{R}^N}
    \frac{|u_n|^{p_s^*(\beta)}}
         {|y|^{\beta}} \,dy.
\end{align*}
Therefore, the sequence 
$\{\tilde{u}_n\}_n\subset D^{s,p}(\mathbb{R}^N)$ is also a 
$(PS)_c$ sequence for the functional $\Phi$ and verifies equality~\eqref{norma.seq.u.tilde=epsilon}.
The lemma is proved. 
\end{proof}

Finally, we can present the proof of our theorem.
\begin{proof}[Proof of Theorem~\ref{teo1}]
Let $\{u_n\}_n \subset D^{s,p}(\mathbb{R}^{N})$ be a Palais-Smale sequence $(PS)_c$ for the functional $\Phi$ at the level $c \in (0,c^*)$.
We claim that this sequence is bounded in the function space $D^{s,p}(\mathbb{R}^{N})$;
after a passage to a subsequence if necessary, still denoted in the same way,
$u_n \rightharpoonup \tilde{u}$ weakly in $D^{s,p}(\mathbb{R}^{N})$ as $n \rightarrow +\infty$ and 
$\tilde{u}$ is a weak solution to problem~\eqref{1.1}.

Indeed, first we have to show that the sequence 
$\{\tilde{u}_n\}_n \subset D^{s,p}(\mathbb{R}^{N})$ is bounded in the function space $D^{s,p}(\mathbb{R}^N)$. 
To accomplish this, we use the pair of 
inequalities~\eqref{eq4fp} and the fact that
$\{\tilde{u}_n\}_n$ is a $(PS)_c$ sequence.
It follows that there exist positive real numbers
$C_1$ and $C_2$ such that
\begin{align*}
\Phi(\tilde{u}_n)
-\frac{1}{p_s^*(\alpha)}
\langle \Phi'(\tilde{u}_n),\tilde{u}_n\rangle
& 
= \left( \frac{1}{p}-\frac{1}{p_s^*(\alpha)} \right)
\Vert \tilde{u}_n\Vert^p \\
& \qquad + \left( \frac{1}{p_s^*(\alpha)}
- \frac{1}{p_s^*(\beta)} \right)
\int_{\mathbb{R}^N}
\frac{|\tilde{u}_n|^{p_s^*(\beta)}}
     {\vert x \vert^{\beta}} \,dx \\
& \leqslant C_1 + C_2 \Vert\tilde{u}_n\Vert +o_n(1);
\end{align*}
thus,
\begin{align*}
\left(\frac{1}{p}-\frac{1}{p_s^*(\alpha)}\right)
\Vert \tilde{u}_n \Vert^p
& \leqslant C_1+C_2\Vert\tilde{u}_n\Vert.
\end{align*}

This inequality assures us that 
$\{\Vert\tilde{u}_n\Vert\}_n$ is bounded. 
Using once more the pair of inequalities~\eqref{eq4fp},
we deduce that the sequence 
$\{\tilde{u}_n\}_n \subset D^{s,p}(\mathbb{R}^N)$ is bounded.
So, after a passage to a subsequence if necessary, still denoted in the same way, 
there exists $\tilde{u}\in D^{s,p}(\mathbb{R}^N)$ 
such that $\tilde{u}_n \rightharpoonup \tilde{u}$ 
weakly in $D^{s,p}(\mathbb{R}^N)$ as $n \to +\infty$. 

Notice that if $\tilde{u}\equiv 0$, then Proposition~\ref{2} guarantees that 
\begin{alignat*}{2}
\varlimsup_{n\rightarrow+\infty}
\int_{B_1(0)}
\frac{|\tilde{u}_n|^{p_s^*(\alpha)}}
     {\vert x \vert^{\alpha}} \,dx
& = 0
& \quad \text{or} \quad 
\varlimsup_{n\rightarrow+\infty}
\int_{B_1(0)}
\frac{|\tilde{u}_n|^{p_s^*(\alpha)}}
     {\vert x \vert^{\alpha}} \,dx
& \geqslant \varepsilon_0.
\end{alignat*} 
Since $0 < \varepsilon < \frac{\varepsilon_0}{2}$, 
we get a contradiction in both cases in view of equality~\eqref{norma.seq.u.tilde=epsilon}.
We deduce that $\tilde{u} \not \equiv 0$. 

Notice that from the weak convergence
$\tilde{u}_n\rightharpoonup \tilde{u}$ in 
$D^{s,p}(\mathbb{R}^N)$ we get
\begin{align*}
\tilde{u}_n & \rightarrow \tilde{u} \quad
\text{a.~e. } \mathbb{R}^N.
\end{align*}
We also have, 
from inequalities~(\ref{desg.hardy}) and~\eqref{eq4fp} 
and from the definition~\eqref{6}, that the sequence
$\{\tilde{u}_n\}_n \subset D^{s,p}(\mathbb{R}^N)$ is bounded 
in $L^{p_s^*(\beta)}(\mathbb{R}^N,\vert x\vert^{-\beta})$ 
and in $L^{p_s^*(\alpha)}(\mathbb{R}^N;|x|^{-\alpha})$. 
Thus, using a result in Kavian~\cite[Lemme 4.8]{MR1276944},
we deduce that
\begin{align*}
\tilde{u}_n & \rightharpoonup \tilde{u}
\quad \text{weakly in } 
L^{p_s^*(\beta)}(\mathbb{R}^N,\vert x\vert^{-\beta})
\text{ as } n\rightarrow+\infty, \\
\tilde{u}_n & \rightharpoonup \tilde{u}
\quad \text{weakly in } 
L^{p_s^*(\alpha)}(\mathbb{R}^N,|x|^{-\alpha})
\text{ as } n\rightarrow+\infty.
\end{align*}
From these convergences, for an arbitrary test function 
$\varphi \in D^{s,p}(\mathbb{R}^N)$ we have
\begin{align*}
\langle\Phi'(\tilde{u}_n),\varphi\rangle 
& = \displaystyle\int_{\mathbb{R}^{N}}\int_{\mathbb{R}^{N}}
    \frac{J_p(\tilde{u}_n(x)-\tilde{u}_n(y))
         (\varphi(x)-\varphi(y))}
         {\vert x-y\vert^{N+sp}} \,dx\,dy \\
& \,  - \mu \int_{\mathbb{R}^N}
    \frac{J_p(\tilde{u}_n)\varphi}
         {\vert x\vert^{ps}} dx -\int_{\mathbb{R}^N}
    \frac{J_{p_s^*(\beta)}(\tilde{u}_n)\varphi}
         {\vert x\vert^\beta}dx
    - \int_{\mathbb{R}^N}
      \frac{ J_{p_s^*(\alpha)}(\tilde{u}_n)\varphi}
           {\vert x\vert^{\alpha}}dx \\
& = o_n(1) + \displaystyle\int_{\mathbb{R}^{N}}\int_{\mathbb{R}^{N}}
    \frac{J_p(\tilde{u}(x)-\tilde{u}(y))
         (\varphi(x)-\varphi(y))}
         {\vert x-y\vert^{N+sp}} \,dx\,dy \\
& \quad   - \mu \int_{\mathbb{R}^N}
      \frac{J_p(\tilde{u})\varphi}
           {\vert x\vert^{ps}} \,dx 
    -\int_{\mathbb{R}^N}
    \frac{J_{p_s^*(\beta)}(\tilde{u})\varphi}
         {\vert x\vert^{\beta}} \,dx
    - \int_{\mathbb{R}^N}
      \frac{J_{p_s^*(\alpha)}(\tilde{u})\varphi}
           {\vert x\vert^{\alpha}} \,dx   \\
& = \langle\Phi'(\tilde{u}),\varphi\rangle + o_n(1).
\end{align*}
Since $\{\tilde{u}_n\}_n \subset
D^{s,p}(\mathbb{R}^N) $
is a $(PS)_c$ sequence, it follows that
$\langle\Phi'(\tilde{u}_n),\varphi\rangle\rightarrow 0\;\mbox{quando}\;n\rightarrow+\infty.
$; and since the test function 
$\varphi \in D^{s,p}(\mathbb{R}^N)$ is arbitrary, it follows that
\begin{align*}
\langle \Phi'(\tilde{u}),\varphi \rangle
& = 0 \quad \text{for all } \varphi\in D^{s,p}(\mathbb{R}^N).
\end{align*}
We conclude that $\tilde{u} \in D^{s,p}(\mathbb{R}^N)$
is a nontrivial weak solution to problem~\eqref{1.1}.
\end{proof}

\section{Extremals for the Sobolev inequality}
\label{sec:extremals}
In this section we show the Theorem~\ref{teo2}.
\begin{prop}
The constant $K(\mu,\alpha)$ defined in~\eqref{6} is attained
by a nontrivial function 
$\tilde{u} \in D^{s,p}(\mathbb{R}^N)$
\end{prop}
\begin{proof}
For the case where
$\mu=0$ and $\alpha=0$ we refer the reader to the paper by
Brasco, Mosconi and Squassina~\cite{MR3461371};
and for the case where $\mu=0$ and $\alpha\in (0,sp)$ we refer the reader to the paper by 
Marano and Mosconi~\cite{doi:10.1142/S0219199718500281}. 
Here we consider the case where
$0 < \mu < \mu_H$ and $\alpha\in (0,sp)$. 
Let the functional
$I_{\mu,\alpha} \colon D^{s,p}(\mathbb{R}^N)
\longrightarrow \mathbb{R}$ 
be given by
\begin{align*}
I_{\mu,\alpha}(u)
& \coloneqq 
\frac{\Vert u\Vert^p}
{\left( \displaystyle\int_{\mathbb{R}^N}
\frac{|u|^{p_s^*(\alpha)}}{|x|^{\alpha}} 
\,dx \right)^{\frac{p}{p_s^*(\alpha)}}}.
\end{align*}
Notice that 
\begin{align*}
\int_{\mathbb{R}^N}
\frac{|u|^{p_s^*(\alpha)}}{|x|^{\alpha}} \,dx
& = \int_{\mathbb{R}^N}
    \frac{|u|^{p_s^*(\alpha)\varepsilon}
         |u|^{p_s^*(\alpha)(1-\varepsilon)}}
         {|x|^{\alpha}} \,dx
    \quad \text{for all } \varepsilon\in\mathbb{R}.
\end{align*}

Now we choose 
$\varepsilon 
= \frac{p_s^*\left(p_s^*(\alpha)-p\right)}
       {p_s^*(\alpha)(p_s^*-p)}
= \frac{N(sp-\alpha)}{s(N-\alpha)}$;
thus, 
$1-\varepsilon=\frac{p(p_s^*-p_s^*(\alpha))}{p_s^*(\alpha)(p_s^*-p)}=\frac{\alpha(N-sp)}{sp(N-\alpha)}>0$. 
Using H\"{o}lder inequality with exponents 
$r=\frac{sp}{sp-\alpha}$ and 
$r'=\frac{sp}{\alpha}$, it follows that
\begin{align*}
\int_{\mathbb{R}^N}
\frac{|u|^{p_s^*(\alpha)}}{|x|^{\alpha}} \,dx 
& = \int_{\mathbb{R}^N}
    \left( |u|^{\frac{N(sp-\alpha)}{s(N-ps)}} \right)    
    \left( \frac{|u|^{\frac{\alpha}{s}}}
                {|x|^{\alpha}}
    \right) \,dx \\
& \leqslant 
  \left( \int_{\mathbb{R}^N}|u|^{p_s^*} \,dx
  \right)^{\frac{sp-\alpha}{sp}}
  \left(\int_{\mathbb{R}^N}\frac{|u|^p}{|x|^{ps}} \,dx
  \right)^{\frac{\alpha}{ps}}.
\end{align*} 
Moreover, by the Hardy-Sobolev inequality, we have
\begin{align*}
\int_{\mathbb{R}^N}
\frac{|u|^{p_s^*(\alpha)}}{|x|^{\alpha}} \,dx 
& \leqslant \Vert u\Vert_{p_s^*}^{p_s^*\frac{ps-\alpha}{ps}}
\left( \int_{\mathbb{R}^N}\frac{|u|^p}{|x|^{ps}} \,dx\right)^{\frac{\alpha}{ps}}\\
& \leqslant
\left( \frac{[u]_{s,p}}
            {S_p^p}\right)^{\frac{N(ps-\alpha)}{s(N-ps)}}
\left( \frac{[u]^p_{s,p}}
            {\mu_H}\right)^{\frac{\alpha}{ps}} \\
& = \left( \frac{1}{\mu_H} \right)^{\frac{\alpha}{ps}}
    \left( \frac{1}{S_p^p}
    \right)^{\frac{N(ps-\alpha)}{s(N-sp)}}[u]_{s,p}^{\frac{N(ps-\alpha)}{s(N-sp)}
    + \frac{\alpha}{s}} \\
& = C(N,p,\mu,s,\alpha) [u]_{s,p}^{\frac{p(N-\alpha)}{N-ps}}
\end{align*}
where $S_p$ is the best constant of the embedding 
$D^{s,p}(\mathbb{R}^N) 
\hookrightarrow L^{p_s^*}(\mathbb{R}^N)$ 
and $\mu_H$ is defined in~\eqref{muH}. 
These inequalities imply that
\begin{align*}
\left(\int_{\mathbb{R}^N}\frac{|u|^{p_s^*(\alpha)}}{|x|^{\alpha}}dx\right)^{\frac{p}{p_s^*}}
& \leqslant C(N,p,\mu,s,\alpha)^{\frac{p}{p_s^*}}[u]_{s,p}^p. 
\end{align*}
We conclude that there exists a positive constant,
still denoted by $C= C(N,p,\mu,s,\alpha)$, such that 
\begin{align*}
\frac{1}{C}
& \leqslant
\dfrac{ [u]_{s,p}^p}
{\left(\displaystyle\int_{\mathbb{R}^N}
 \frac{|u|^{p_s^*(\alpha)}}
 {|x|^{\alpha}} \,dx \right)^{\frac{p}{p_s^*(\alpha)}}}
 \quad \text{for all } u \in D^{s,p}(\mathbb{R}^N)\setminus\{0\}.
\end{align*}

By the pair of inequalities~\eqref{eq4fp} we infer that the functional $I_{\mu,\alpha}$ is well defined and is bounded from below by a positive constant; 
so, $\inf I_{\mu,\alpha} > 0$. 
Note that 
$I_{\mu,\alpha} \in C^1(D^{s,p}(\mathbb{R}^N),\mathbb{R})$
and we can apply Ekeland variational principle~\cite{MR0346619} to the functional 
$I_{\mu,\alpha}(u)$ 
to guarantee the existence of a minimizing sequence
$\{u_n\}_n\subset D^{s,p}(\mathbb{R}^N)$ with the additional properties
\begin{align*}
\int_{\mathbb{R}^N}
  \frac{|u_n|^{p_s^*(\alpha)}}
       {|x|^{\alpha}} \,dx
& = 1 \quad \text{for all } n\in\mathbb{N}
\end{align*}
and also
\begin{align*}
I_{\mu,\alpha}(u_n) 
& \rightarrow 
\inf_{\substack{u \in D^{s,p}(\mathbb{R}^N) \\ u \neq 0}}
I_{\mu,\alpha}(u) \coloneqq \frac{1}{K(\mu,\alpha)}
\quad \textup{as } n \to +\infty; \\ 
I_{\mu,\alpha}'(u_n)
& \rightarrow 0\quad \text{in } (D^{s,p}(\mathbb{R}^N))'.
\end{align*}
Now we consider two auxiliary functionals
$J,G \colon D^{s,p}(\mathbb{R}^N) \longrightarrow \mathbb{R}$ defined by
\begin{alignat*}{2}
J(u) 
& \coloneqq \frac{\Vert u\Vert^p}{p}
& \quad \text{and}\quad 
G(u) 
&\coloneqq \frac{1}{p_s^*(\alpha)}
  \int_{\mathbb{R}^N}
  \frac{|u|^{p_s^*(\alpha)}}{|x|^{\alpha}} \,dx.
\end{alignat*}
\begin{claim} The following statements are valid.
\begin{enumerate}
\item $ J(u_n) \rightarrow \dfrac{1}{pK(\mu,\alpha)}
\quad \text{as } n \rightarrow +\infty $.
\item $ J'(u_n) - \dfrac{1}{K(\mu,\alpha)}G'(u_n)
 \rightarrow 0$
in $(D^{s,p}(\mathbb{R}^N))'$
as $n\rightarrow+\infty$.
\end{enumerate}
\end{claim}

The proof of the first item is imediate from the definitions involved. To prove the second item we consider
$v\in D^{s,p}(\mathbb{R}^N)$ and note that
\begin{align*}
J'(u_n)v 
& = \int_{\mathbb{R}^N}\int_{\mathbb{R}^N}
    \frac{J_p(u_n(x)-u_n(y))(v(x)-v(y))}
         {|x-y|^{N+sp}}dx\,dy \\
& \qquad  - \mu \int_{\mathbb{R}^N}\frac{J_p(u_n)v}{|x|^{ps}}dx,
\end{align*}
and
\begin{align*}
G'(u_n)v
& = \int_{\mathbb{R}^N}
    \frac{|u_n|^{p_s^*(\alpha)-2}u_nv}
         {|x|^{\alpha}} \,dx.
\end{align*}
Thus,
\begin{align*}
J'(u_n)v &- \frac{1}{K(\mu,\alpha)}\,G'(u_n)v
 = \int_{\mathbb{R}^N}\int_{\mathbb{R}^N}
    \frac{J_p(u_n(x)-u_n(y))(v(x)-v(y))}
         {|x-y|^{N+sp}} \,dx\,dy \\
& \qquad   - \mu \int_{\mathbb{R}^N}
    \frac{J_p(u_n)v}{|x|^{ps}} \,dx - \frac{1}{K(\mu,\alpha)}
    \int_{\mathbb{R}^N}
    \frac{J_{p_s^*(\alpha)}(u_n)v}{|x|^{\alpha}} \,dx.
\end{align*}
We remark that
\begin{align*}
I_{\mu,\alpha}'(u_n)v =\frac{p(J'(u_n)v)}{\left( \displaystyle \int_{\mathbb{R}^N}
    \frac{|u_n|^{p_s^*(\alpha)}}{|x|^{\alpha}} \,dx
    \right)^{\frac{p}{p_s^*(\alpha)}}}
    -\frac{p \Vert u_n\Vert^{p}G'(u_n)v}
    {\left(\displaystyle
     \int_{\mathbb{R}^N}
     \frac{|u_n|^{p_s^*(\alpha)}}{|x|^{\alpha}} \,dx
    \right)^{\frac{p}{p_s^*(\alpha)}+1}}.
\end{align*}
Since $\Vert u_n\Vert^p\rightarrow \frac{1}{K(\mu,\alpha)}$ as $n\rightarrow+\infty$ and 
$\displaystyle\int_{\mathbb{R}^N}
\frac{|u_n|^{p_s^*(\alpha)}}{|x|^{\alpha}}dx
= 1$ for all $n\in\mathbb{N}$,
we obtain
\begin{align*}
I_{\mu,\alpha}'(u_n)v
& = p \left( J'(u_n)v 
   - \left(\frac{1}{K(\mu,\alpha)} + o_n(1)\right)
   G'(u_n)v\right).
\end{align*}
And by the fact that 
$I_{\mu,\alpha}'(u_n)v \rightarrow 0$ 
as $n\rightarrow+\infty$ for an arbitrary function
$v\in D^{s,p}(\mathbb{R}^N)$, 
we deduce that
\begin{align*}
J'(u_n)v 
- \frac{1}{K(\mu,\alpha)}G'(u_n)v
& \rightarrow 0
\quad \text{for all } v \in D^{s,p}(\mathbb{R}^N)
\textup{ as } n\rightarrow+\infty.
\end{align*}
Therefore, 
\begin{align}
J'(u_n) - \frac{1}{K(\mu,\alpha)}G'(u_n)
& \rightarrow 0 
\quad \textup{in } (D^{s,p}(\mathbb{R}^N))'
 \textup{ as } n\rightarrow+\infty.\label{eq.para.tomar.funcao.teste}
\end{align}
This concludes the proof of the claim.

Now we define the Levy concentration function
$Q \colon \mathbb{R}_+ \to \mathbb{R}$ 
associated to  
$\frac{|u_n|^{p_s^*(\alpha)}}{|x|^\alpha}$ by
\begin{align*}
Q(r) 
& \coloneqq \int_{B_r(0)}
  \frac{|u_n|^{p_s^*(\alpha)}}{|x|^\alpha} \,dx.
\end{align*}

Using the continuity of the function $Q$ and since $\int_{\mathbb{R}^N}\frac{|u_n|^{p_s^*(\alpha)}}{|x|^\alpha}dx=1$,
passing to a subsequence if necessary, still denoted in the same way, for every natural number $n\in\mathbb{N}$
there exists a positive real number $r_n>0$ such that 
\begin{align*}
Q(r_n) & = \int_{B_{r_n}(0)}\frac{|u_n|^{p_s^*(\alpha)}}{|x|^\alpha}dx=\frac{1}{2}
\quad \text{for all } n\in\mathbb{N}.
\end{align*}

As we have already done, we consider the sequence
$\{\tilde{u}_n\}_n\subset D^{s,p}(\mathbb{R}^N)$ 
defined in~\eqref{def.u.tilde}. We note that 
$I_{\mu,\alpha}(\tilde{u}_n) = I_{\mu,\alpha}(u_n)$ 
because the several integrals present in the functional $I_{\mu,\alpha}$ are invariant under the conformal transformations defined in~\eqref{def.u.tilde}. 
In this way,
$\{\tilde{u}_n\}_n \subset D^{s,p}(\mathbb{R}^N)$ is also 
a minimizing sequence for the functional $I_{\mu,\alpha}$;
moreover,
\begin{align}
\int_{B_1(0)}\frac{|\tilde{u}_n|^{p_s^*(\alpha)}}
{|x|^\alpha}dx
& = \frac{1}{2} \quad \text{for all } n\in\mathbb{N}.\label{norma.u.tilde.na.bola=1/2}
\end{align} 
We also note that 
$\Vert \tilde{u}_n\Vert^p=\frac{1}{K(\mu,\alpha)}+o_n(1)$;
so, by inequalities~\eqref{eq4fp} it follows that the sequence $\{\tilde{u}_n\}_n \subset D^{s,p}(\mathbb{R}^N)$ is bounded. We deduce that, up to a passage to a subsequence, there exists a function 
$\tilde{u}\in D^{s,p}(\mathbb{R}^N)$ such that,
as $n\rightarrow+\infty$ we have
\begin{align*}
\tilde{u}_n 
& \rightharpoonup \tilde{u} 
\quad \text{weakly in } D^{s,p}(\mathbb{R}^N) \\
\tilde{u}_n
& \rightarrow \tilde{u}
\quad \text{strongly in }L^q_{\operatorname{loc}}(\mathbb{R}^N),
\quad \text{for all } q\in [1,p_s^*).
\end{align*}
 
Our goals now are to prove that 
$\tilde{u}\not\equiv 0$ and that 
$\tilde{u}$ is a minimizer for the functional 
$I_{\mu,\alpha}$. To accomplish the first goal we argue by contradiction and suppose that 
$\tilde{u}\equiv 0$. So, we have
\begin{align}
\tilde{u}_n 
& \rightharpoonup 0 
\quad \text{weakly in } D^{s,p}(\mathbb{R}^N) \\
\tilde{u}_n
& \rightarrow 0
\quad \text{strongly in }L^q_{\operatorname{loc}}(\mathbb{R}^N)
\text{ for all } q\in [1,p_s^*).
\end{align}
  
Now we set $0 < \delta < 1$ and define 
$B_\delta(0) \coloneqq 
\{x\in\mathbb{R}^N;|x|\leqslant\delta<1\}$. 
We also consider a cut-off function 
$\eta\in C^{\infty}_0(\mathbb{R}^N)$ 
such that 
\begin{alignat}{3}
& \eta \equiv 1 
& \quad\text{in } B_\delta(0);
\qquad & \eta \equiv 0 
  \quad\text{in } \mathbb{R}^N\setminus B_1(0);
\qquad & 0 \leqslant\eta\leqslant 1 
& \quad\text{in } \mathbb{R}^N.
\end{alignat}

Using a result in 
Brasco, Squassina and Yang~\cite[Lemma A.1]{MR3732174},
we can also consider 
$\eta^p\tilde{u}_n$ as test function 
in~\eqref{eq.para.tomar.funcao.teste};
so,
\begin{align}
&\int_{\mathbb{R}^N}\int_{\mathbb{R}^N}
\frac{J_p(\tilde{u}_n(x)-\tilde{u}_n(y))
(\eta^p\tilde{u}_n(x)-\eta^p\tilde{u}_n(y))}{|x-y|^{N+sp}}
\,dx\,dy
-\mu \int_{\mathbb{R}^N}
\frac{|\eta\tilde{u}_n|^p}{|x|^{ps}}\,dx \nonumber\\
&\qquad\qquad\qquad = \frac{1}{K(\mu,\alpha)}
\int_{\mathbb{R}^N}
\frac{|\tilde{u}_n|^{p_s^*(\alpha)}\eta^p}
     {|x|^{\alpha}} \,dx + o_n(1).   \label{eta.un.como.fcao.teste} 
\end{align}

As we have already seen, 
if $\tilde{u}_n\rightharpoonup 0$ weakly in 
$D^{s,p}(\mathbb{R}^N)$ 
and $\eta\in C^{\infty}_0(\mathbb{R}^N)$,
then
\begin{align*}
\int_{\mathbb{R}^{N}}\int_{\mathbb{R}^{N}}
\frac{\eta^p(x)|\tilde{u}_n(x)-\tilde{u}_n(y)|^p}
     {|x-y|^{N+sp}} \,dx\,dy
& =[\eta\tilde{u}_n]_{s,p}^p+o_n(1) 
\quad\text{as } n \rightarrow +\infty\\
\shortintertext{and}
\int_{\mathbb{R}^{N}}\int_{\mathbb{R}^{N}}
\frac{|\tilde{u}_n(y)||\eta^p(x)-\eta^p(y)|^p}
     {|x-y|^{N+sp}} \,dx\,dy 
& \rightarrow 0
\quad \text{as } n\rightarrow +\infty.
\end{align*}
Both these estimates together with the 
estimate~\eqref{eta.un.como.fcao.teste} and 
H\"{o}lder inequality with exponents
$p_s^*(\alpha)$
and
$p_s^*(\alpha)/(p_s^*(\alpha)-p)$
allow us to deduce that
\begin{align}
o_n(1)&+\Vert\eta\tilde{u}_n\Vert^p
 = \frac{1}{K(\mu,\alpha)}
    \int_{\mathbb{R}^N}
    \frac{|\tilde{u}_n|^{p_s^*(\alpha)}\eta^p}
         {|x|^\alpha} \,dx 
    \nonumber \\
& \leqslant \frac{1}{K(\mu,\alpha)}
  \int_{B_1(0)}
  \frac{|\tilde{u}_n|^{p_s^*(\alpha)}}
       {|x|^\alpha} \,dx
   \nonumber \\
& = \frac{1}{K(\mu,\alpha)}
    \left( \int_{B_1(0)}
           \frac{|\tilde{u}_n|^{p_s^*(\alpha)}}
                {|x|^\alpha} \,dx 
    \right)^{\frac{p}{p_s^*(\alpha)}}
    \left( \int_{B_1(0)}
           \frac{|\tilde{u}_n|^{p_s^*(\alpha)}}
                {|x|^\alpha} \,dx
    \right)^{\frac{ps-\alpha}{N-\alpha}}
    \nonumber \\
& \leqslant 
\frac{1}{K(\mu,\alpha)}
\left(\frac{1}{2}\right)^{\frac{ps-\alpha}{N-\alpha}}
\left( \int_{B_1(0)}
       \frac{|\tilde{u}_n|^{p_s^*(\alpha)}}
            {|x|^\alpha} \,dx 
\right)^{\frac{p}{p_s^*(\alpha)}}.\label{est.norma.eta.u.tilde}
\end{align}
where in the last passage we used equality~\eqref{norma.u.tilde.na.bola=1/2}.

We note that 
\begin{align*}
\Vert\tilde{u}_n\Vert_{L^{p_s^*(\alpha)}(B_1(0),|x|^{-\alpha})}
& = \Vert\eta\tilde{u}_n+(1-\eta)\tilde{u}_n\Vert_{L^{p_s^*(\alpha)}(B_1(0),|x|^{-\alpha})}\\
&\leqslant \Vert\eta\tilde{u}_n\Vert_{L^{p_s^*(\alpha)}(B_1(0),|x|^{-\alpha})}\\
&\qquad+\Vert(1-\eta)\tilde{u}_n\Vert_{L^{p_s^*(\alpha)}(B_1(0),|x|^{-\alpha})};
\end{align*}
and since
$\eta\equiv 1$ in $B_\delta(0)$, it follows that
\begin{align*}
\Vert(1-\eta)\tilde{u}_n\Vert_{L^{p_s^*(\alpha)}(B_1(0),|x|^{-\alpha})}^{p_s^*(\alpha)}
& \leqslant C\int_{B_1(0)\setminus B_\delta(0)}
|\tilde{u}_n|^{p_s^*(\alpha)}dx.
\end{align*}
We also have 
$\tilde{u}_n\rightarrow 0$ strongly in 
$L^q_{loc}(\mathbb{R}^N)$ for $q\in [1,p_s^*)$;
and since $p_s^*(\alpha)<p_s^*$, it follows that
\begin{align*}
\Vert(1-\eta)\tilde{u}_n\Vert_{L^{p_s^*(\alpha)}(B_1(0),|x|^{-\alpha})}^{p_s^*(\alpha)}
& = o_n(1).
\end{align*}
Consequently,
\begin{align}
\Vert\tilde{u}_n\Vert_{L^{p_s^*(\alpha)}(B_1(0),|x|^{-\alpha})}
& = \left( \int_{B_1(0)}
       \frac{|\tilde{u}_n|^{p_s^*(\alpha)}}
            {|x|^\alpha} \,dx
\right)^{\frac{p}{p_s^*(\alpha)}}\nonumber\\
& \leqslant
\left( \int_{\mathbb{R}^N}
       \frac{|\eta\tilde{u}_n|^{p_s^*(\alpha)}}
            {|x|^\alpha} \,dx
\right)^{\frac{p}{p_s^*(\alpha)}}
+ o_n(1).\label{est.superior.u.tilde.eta.u.tilde}
\end{align} 
Substituting inequality~\eqref{est.superior.u.tilde.eta.u.tilde} 
into the estimate~\eqref{est.norma.eta.u.tilde}
yields 
\begin{align}
\Vert\eta\tilde{u}_n\Vert^p
& \leqslant
  \frac{1}{K(\mu,\alpha)}
  \left( \frac{1}{2} \right)^{\frac{ps-\alpha}{N-\alpha}}
  \left( \int_{\mathbb{R}^N}
         \frac{|\eta\tilde{u}_n|^{p_s^*(\alpha)}}
              {|x|^\alpha} \,dx
  \right)^{\frac{p}{p_s^*(\alpha)}}
  + o_n(1).\label{est.superior.eta.un}
\end{align}
On the other hand, by the definition~\eqref{6} of the constant $K(\mu,\alpha)$, we obtain
\begin{align}
\frac{1}{K(\mu,\alpha)}
\left( \int_{\mathbb{R}^N}
       \frac{|\eta\tilde{u}_n|^{p_s^*(\alpha)}}
            {|x|^\alpha} \,dx
\right)^{\frac{p}{p_s^*(\alpha)}}
& \leqslant \Vert\eta\tilde{u}_n\Vert^p.\label{est.inferior.eta.un}
\end{align}

Combining 
inequalities~\eqref{est.superior.eta.un} 
and~\eqref{est.inferior.eta.un} we arrive at
\begin{align*}
\frac{1}{K(\mu,\alpha)}
\left( 1- \left(\frac{1}{2}\right)^{\frac{ps-\alpha}{N-\alpha}}\right)
\left(\int_{\mathbb{R}^N}\frac{|\eta\tilde{u}_n|^{p_s^*(\alpha)}}{|x|^\alpha}dx\right)^{\frac{p}{p_s^*(\alpha)}}
& \leqslant o_n(1).
\end{align*}

Using the assumptions $\alpha\in (0,sp)$ and $N>sp$
it follows from the previous inequality that
\begin{align*}
\int_{\mathbb{R}^N}\frac{| \eta \tilde{u}_n|^{p_s^*(\alpha)}}{|x|^\alpha}dx
& \rightarrow 0.
\end{align*}
So,this inequality together with estimate~\eqref{est.superior.u.tilde.eta.u.tilde} guarantees that
\begin{align*}
\int_{B_1(0)}
\frac{|\tilde{u}_n|^{p_s^*(\alpha)}}
     {|x|^\alpha} \,dx
& \rightarrow 0 \quad \text{as } n\rightarrow+\infty.
\end{align*}
But this is a contradiction with 
equality~\eqref{norma.u.tilde.na.bola=1/2}. 
As a result, we deduce that $\tilde{u}\not\equiv 0$. 

It remains to show that the weak limit
$\tilde{u}\in D^{s,p}(\mathbb{R}^N)$ is in fact a minimizer to $\frac{1}{K(\mu,\alpha)}$ and that 
$\int_{\mathbb{R}^N}\frac{|\tilde{u}|^{p_s^*(\alpha)}}{|x|^\alpha}dx=1$. 
To do this, for every natural number 
$n\in\mathbb{N}$ we define
$\theta_n \coloneqq \tilde{u}_n-\tilde{u}$. 
Hence, applying Brezis-Lieb lemma~\cite[Theorem~1.1]{MR699419} we obtain
\begin{align*}
1 
& = \int_{\mathbb{R}^N}
    \frac{|\tilde{u}_n|^{p_s^*(\alpha)}}
         {|x|^\alpha} \,dx
  = \int_{\mathbb{R}^N}
    \frac{|\tilde{u}|^{p_s^*(\alpha)}}
         {|x|^\alpha} \,dx
  + \int_{\mathbb{R}^N}
    \frac{|\theta_n|^{p_s^*(\alpha)}}
         {|x|^\alpha} \,dx 
  + o_n(1).
\end{align*}
From this estimate we get the inequalities
\begin{alignat*}{2}
0 
& \leqslant 
  \int_{\mathbb{R}^N}
  \frac{|\tilde{u}|^{p_s^*(\alpha)}}
       {|x|^\alpha} \,dx
  \leqslant 1
& \quad\text{and}\quad 
0
& \leqslant
  \int_{\mathbb{R}^N}
  \frac{|\theta_n|^{p_s^*(\alpha)}}
       {|x|^\alpha} \,dx
  \leqslant 1.
\end{alignat*}
Using the fact that 
$\theta_n\rightharpoonup 0$ 
weakly in $D^{s,p}(\mathbb{R}^N)$ 
as $n\rightarrow+\infty$, aplying a result by 
Brasco, Squassina 
and Yang~\cite[Lemma 2.2]{MR3732174}, 
we infer that
\begin{align*}
\Vert\tilde{u}_n\Vert^p
& =\Vert\tilde{u}\Vert^p+\Vert\theta_n\Vert^p+o_n(1).
\end{align*}
In this way, using
estimate~\eqref{eq.para.tomar.funcao.teste}
and the definition~\eqref{6} of the constant 
$K(\mu,\alpha)$ we get
\begin{align*}
o_n(1)
& = \Vert\tilde{u}_n\Vert^p
  - \frac{1}{K(\mu,\alpha)}
    \int_{\mathbb{R}^N}
    \frac{|\tilde{u}_n|^{p_s^*(\alpha)}}
         {|x|^\alpha} \,dx \\
& = \left(\Vert\tilde{u}\Vert^p
  - \frac{1}{K(\mu,\alpha)}
    \int_{\mathbb{R}^N}
    \frac{|\tilde{u}|^{p_s^*(\alpha)}}
         {|x|^\alpha} \,dx
    \right) \\
& \qquad 
  + \left( \Vert\theta_n\Vert^p
          -\frac{1}{K(\mu,\alpha)}
          \int_{\mathbb{R}^N}
          \frac{|\theta_n|^{p_s^*(\alpha)}}
          {|x|^\alpha} \,dx
    \right)
  + o_n(1)\\
& \geqslant
  \frac{1}{K(\mu,\alpha)}
  \left\{
      \left( \int_{\mathbb{R}^N}
             \frac{|\tilde{u}|^{p_s^*(\alpha)}}
                  {|x|^\alpha} \,dx
      \right)^{\frac{p}{p_s^*(\alpha)}}
    - \int_{\mathbb{R}^N}
      \frac{|\tilde{u}|^{p_s^*(\alpha)}}
           {|x|^\alpha} \,dx 
  \right\}\\
& \quad 
  + \frac{1}{K(\mu,\alpha)}
    \left\{
       \left( \int_{\mathbb{R}^N}
              \frac{|\theta_n|^{p_s^*(\alpha)}}
                   {|x|^\alpha} \,dx
       \right)^{\frac{p}{p_s^*(\alpha)}}
     - \int_{\mathbb{R}^N}
       \frac{|\theta_n|^{p_s^*(\alpha)}}
            {|x|^\alpha} \,dx
   \right\} + o_n(1)\\
& \eqqcolon \frac{1}{K(\mu,\alpha)}\, A 
+ \frac{1}{K(\mu,\alpha)}\, B + o_n(1).
\end{align*}
Clearly, $A+B=o_n(1)$;
and using the fact that $\frac{p}{p_s^*(\alpha)}\in (0,1)$
and that both integrals
$\int_{\mathbb{R}^N}
\frac{|\tilde{u}|^{p_s^*(\alpha)}}{|x|^\alpha}\,dx$
and
$\int_{\mathbb{R}^N}
\frac{|\theta_n|^{p_s^*(\alpha)}}{|x|^\alpha}\,dx$
take their values in the closed interval $[0,1]$,
we deduce that $A,B\geqslant 0$ and that
$B=-A+o_n(1)\geqslant 0$; 
this means that $A=0$ and that $B=o_n(1)$, that is,
\begin{align*}
\left(\int_{\mathbb{R}^N}
\frac{|\tilde{u}|^{p_s^*(\alpha)}}
     {|x|^\alpha} \,dx
\right)^{\frac{p}{p_s^*(\alpha)}}
& = \int_{\mathbb{R}^N}
    \frac{|\tilde{u}|^{p_s^*(\alpha)}}{|x|^\alpha} \,dx. 
\end{align*}
We have already seen that $\tilde{u}\not\equiv 0$;
so, the previous equality implies that
\begin{align*}
\int_{\mathbb{R}^N}\frac{|\tilde{u}|^{p_s^*(\alpha)}}{|x|^\alpha}dx & = 1.
\end{align*}

Using again the
estimate~\eqref{eq.para.tomar.funcao.teste}
and the fact that 
$\tilde{u}_n\rightharpoonup\tilde{u}$ 
weakly in $D^{s,p}(\mathbb{R}^N)$ 
as $n\rightarrow+\infty$, it follows that,
\begin{align*}
J'(\tilde{u}_n)\tilde{u}
- \frac{1}{K(\mu,\alpha)}\, G'(\tilde{u}_n)\tilde{u}
& \rightarrow J'(\tilde{u})\tilde{u}
- \frac{1}{K(\mu,\alpha)}\, G'(\tilde{u})\tilde{u}
= 0
\end{align*}
as $n\rightarrow +\infty$.

Finally, we conclude that
\begin{align*}
\Vert\tilde{u}\Vert^p
& = \frac{1}{K(\mu,\alpha)}
\int_{\mathbb{R}^N}
\frac{|\tilde{u}|^{p_s^*(\alpha)}}
     {|x|^\alpha} \,dx
= \frac{1}{K(\mu,\alpha)},
\end{align*}
that is, the best constant $\frac{1}{K(\mu,\alpha)}$
is attained by a nontrivial function  
$\tilde{u} \in D^{s,p}(\mathbb{R}^{N})$.
This concludes the proof of the proposition.
\end{proof}

\begin{rem}
If $\mu \leqslant 0$, then $1/K(\mu,0)=1/K(0,0)$. Therefore there is no extremal for $1/K(\mu,\alpha)$ when $\mu\leqslant 0$.
\end{rem}
As we have $\mu \leqslant 0$, clearly 
$1/K(\mu,0) \geqslant 1/K(0,0)$. We consider a function $w\in D^{1,2}(\mathbb{R}^N)\setminus \{0\}$ for which $1/K(0,0)$ is attained. For the existence of such a function we refer Brasco, Mosconi and Squassina~\cite{MR3461371}. Now for $\delta\in \mathbb{R}$ and $\overline{x}\in\mathbb{R}^N$ we define the function $w_\delta=w(x-\delta\overline{x})$ for $x\in\mathbb{R}^N$. Then, by changing variables we get
\begin{align*}
\frac{1}{K(\mu,0)}\leqslant I_\delta=\dfrac{\Vert w_\delta \Vert^2}{\left(\displaystyle\int_{\mathbb{R}^N}\vert w_\delta\vert^{p^*_s}\,dx\right)^{\frac{p}{p^*_s}}}= \dfrac{\Vert w \Vert^2}{\left(\displaystyle\int_{\mathbb{R}^N}\vert w \vert^{p^*_s}\,dx\right)^{\frac{p}{p^*_s}}};
\end{align*}
therefore,
\begin{align*}
\frac{1}{K(\mu,0)}\leqslant \lim_{\delta\rightarrow+\infty}I_\delta= \dfrac{\Vert w \Vert^2}{\left(\displaystyle\int_{\mathbb{R}^N}\vert w \vert^{p^*_s}\,dx\right)^{\frac{p}{p^*_s}}}=\frac{1}{K(0,0)}.
\end{align*}
So, 
\begin{align*}
\frac{1}{K(\mu,0)}=\frac{1}{K(0,0)}.
\end{align*}
We conclude that there is no function that attained the constant $1/K(\mu,0)$ when $\mu<0$. 

\subsection*{Acknowledgement}
The authors would like to express their very great appreciation to Dr. Patrizia Pucci and Dr.~Shaya Shakerian for their valuable and constructive suggestions during this research work.

\bibliographystyle{siam}
\bibliography{biblio_jef}

\end{document}